\documentclass[reqno]{amsart} 

\usepackage[T1]{fontenc}
\usepackage[utf8]{inputenc}
\usepackage[english]{babel}
\usepackage[sc]{mathpazo} 
\usepackage{eulervm} 

\usepackage{etoolbox} 

\usepackage{mathtools} 
\mathtoolsset{showonlyrefs} 

\AtBeginDocument{\def\MR#1{}} 

\usepackage{amssymb,textcomp,mathrsfs,stmaryrd,bm,braket,appendix,commath}

\usepackage[shortlabels]{enumitem}

\usepackage[marginpar]{todo}

\theoremstyle{plain} 
\newtheorem{thm}{Theorem}[section]
\newtheorem*{thm*}{Theorem}
\newtheorem{prop}{Proposition}[section]
\newtheorem*{prop*}{Proposition}
\newtheorem{lem}{Lemma}[section]
\newtheorem*{lem*}{Lemma}

\theoremstyle{definition}
\newtheorem{defn}{Definition}[section]

\theoremstyle{remark}
\newtheorem{rem}{Remark}[section]
\newtheorem*{rem*}{Remark}

\newcommand{\addQEDstyle}[2]{\AtBeginEnvironment{#1}{\pushQED{\qed}\renewcommand{\qedsymbol}{#2}}
\AtEndEnvironment{#1}{\popQED}} 

\addQEDstyle{rem}{$\triangle$}
\addQEDstyle{rem*}{$\triangle$} 

\DeclareMathOperator\GL{GL}
\DeclareMathOperator\Sp{Sp}
\DeclareMathOperator\SO{SO}
\DeclareMathOperator\SU{SU}
\DeclareMathOperator\U{U}

\DeclareMathOperator\Hk{Hk}
\DeclareMathOperator\mHK{Hk(M)}

\DeclareMathOperator\mSp{Sp(M)}
\DeclareMathOperator\mSpo{Sp_0(M)}
\DeclareMathOperator\Tan{T}

\DeclareMathOperator\mpSpo{Sp_0'(M)}

\DeclareMathOperator\mpHKo{Hk_0'(M)}

\DeclareMathOperator\SpanIJK{\mathfrak{su}(2)_{V}}
\DeclareMathOperator\SpanIJKM{\mathfrak{su}(2)_{M}}

\DeclareMathOperator\Id{Id}
\DeclareMathOperator\triv{Tr}
\DeclareMathOperator\End{End}
\DeclareMathOperator\rk{rk}
\DeclareMathOperator\Hom{Hom}

\DeclareMathOperator{\Stab}{Stab}
\DeclareMathOperator\Ad{Ad}

\DeclareMathOperator\Lie{Lie}
\DeclareMathOperator\Iso{Iso}
\DeclareMathOperator\LL{L^2}
\DeclareMathOperator\T{T}
\DeclareMathOperator\vol{vol}
\DeclareMathOperator\Sym{Sym}

\DeclareMathOperator\topp{top}

\DeclareMathOperator\Aut{Aut}

\DeclareMathOperator\AH{AH}

\usepackage{hyperref}
\hypersetup{colorlinks, linkcolor=red, filecolor=green, urlcolor=blue, citecolor=blue}

\author[J. E. Andersen, A. Malus\`a, G. Rembado]{J\o{}rgen Ellegaard Andersen, Alessandro Malus\`{a}, Gabriele Rembado} 

\title[$\Sp(1)$-symmetric hyperk\"ahler  quantisation]{$\Sp(1)$-symmetric hyperk\"ahler  quantisation} 

\address[J. E. Andersen]{Centre for Quantum Mathematics, Danish Institute for Advanced Study, University of Southern Denmark, Campusvej 55, 5230 Odense M, Denmark}
\email[J. E. Andersen]{jea@sdu.dk}

\address[A. Malus\`{a}]{Department of Mathematics, University of Toronto, 40 St. George Street, Toronto, Ontario, Canada}
\email[A. Malus\`a]{amalusa@math.utoronto.ca}

\address[G. Rembado]{Hausdorff Centre for Mathematics, University of Bonn, Endenicher Allee 62, 53115, Bonn, Germany}
\curraddr{Institut Montpelliérain Alexander Grothendieck (IMAG), University of Montpellier, Place Eugène Bataillon 34090 Montpellier}
\email[G. Rembado]{gabriele.rembado@umontpellier.fr}

\begin{document}

\begin{abstract}
    We provide a new general scheme for the geometric quantisation of $\Sp(1)$-symmetric hyperk\"ahler manifolds, considering Hilbert spaces of holomorphic sections with respect to the complex structures in the hyperk\"ahler 2-sphere.
    Under properness of an associated moment map, or other finiteness assumptions, we construct unitary quantum (super) representations of groups acting as certain Riemannian isometries preserving the 2-sphere, and we study their decomposition in irreducible components. 
    We apply this quantisation scheme to hyperk\"ahler vector spaces, the Taub--NUT metric on $\mathbb{R}^4$, moduli spaces of framed $\SU(r)$-instantons on $\mathbb{R}^4$, and partly to the Atiyah--Hitchin manifold of magnetic monopoles in $\mathbb{R}^3$.
\end{abstract}

{\let\newpage\relax\maketitle} 

\setcounter{tocdepth}{1} 
\tableofcontents

\section{Introduction}

The constructions of geometric quantisation offer a recipe for addressing problems related to the quantum mechanics of an object moving in an arbitrary, possibly curved, phase space~\cite{GS77,Woo80}.
The process, abstracting canonical quantisation, is fundamentally based on the structure of a symplectic manifold.
Two of the main goals are to obtain operators subject to commutation relations prescribed by the Poisson bracket, and unitary representations of groups associated to Hamiltonian flows.
However, there are strong limitations to the extent to which these can be achieved in general.
One of the most typical problems is the need of a polarisation, whose existence is generally not guaranteed, nor is its uniqueness ever satisfied.
Further, the choice of a particular polarisation poses serious constraints on which functions and Hamiltonian flows can be quantised.

A common approach to this issue consists in considering instead a \emph{family} of polarisations, parametrised by some smooth manifold.
One then attempts to assemble their corresponding quantum Hilbert spaces into a vector bundle and identify them via the holonomy of some appropriate connection.
In this framework, the natural way to quantise Hamiltonian group actions is by automorphisms of the bundle as a whole rather than of the individual vector spaces.
A group representation is then usually obtained by considering the space of (projectively) flat sections.
The prototypical example of this is the Hitchin connection~\cite{Hit90,ADPW91}, further discussed below.
The latter also has a simple yet interesting adaptation to the case of a symplectic linear space, providing a quantisation of its full symplectic group in the form of a representation of a double cover of it, the metaplectic representation~\cite[Chapter 10]{Woo80}.

Because polarisations on a symplectic manifold often arise as compatible complex structures, it is rather common in geometric quantisation to work with Kähler manifolds~\cite{Cha16,Woo80}.
This approach has been successfully applied to a number of moduli spaces arising from differential and algebraic geometry, representation theory, and mathematical physics. Notable examples include unitary flat connections~\cite{Hit90,ADPW91} and vector bundles on Riemann surfaces, compact coadjoint orbits~\cite{Kir04}, and polygons~\cite{KM,Cha10}.
Unitary flat connections in particular are a good example of the scheme sketched above, as the moduli space comes with a family of Kähler structures parametrised by the Teichmüller space.
The construction of a projectively flat connection in that setting is due to Hitchin~\cite{Hit90} and Axelrod-Della Pietra-Witten~\cite{ADPW91} and was extended to a broader framework in later works~\cite{And06,AG14,AR14}.
The role of flat connections in Chern--Simons theory~\cite{Wit89,Fre95} also motivated further study of the relation between geometric quantisation and other formulations of the theory, including deformation quantisation~\cite{And12,BMS94,KS01,Sch00,Sch01,AM19} and other approaches~\cite{AU1,AU2,AU3,AU4,Las98,AM17}.

In many cases, spaces similar to those above, and related to interesting quantisation problems, come with natural hyperkähler structures rather than just Kähler.
Some of these may be viewed as complexifications of those already mentioned, e.g. flat connections for complex groups~\cite{Wit91}, Higgs bundles~\cite{AGP16}, semisimple/nilpotent (co)adjoint orbits in (dual) complex Lie algebras~\cite{Kro88,Kov96,Biq96}, hyper-polygons, and Nakajima quiver varieties~\cite{Gin09,Nak94}. 
Others, on the other hand, arise independently of an underlying ``real'' version, including for instance the Taub-NUT metrics on $\mathbb{R}^{4}$, moduli spaces of framed $\SU(r)$-instantons and magnetic monopoles, and the Nahm moduli spaces.

Crucially, the triple of complex structures on each of these spaces gives rise to a family of Kähler forms, whose parametrising space comes with its own Kähler structure which identifies it with $\mathbb{C}P^{1}$.
Unlike in the Kähler case, geometric quantisation does not directly apply in this situation because no preferred symplectic structure is given in general.
What is more, many interesting symmetries of such spaces act by hyperkähler rotations, i.e. by permuting the sphere of Kähler structures rather than fixing them individually.
If one of the symplectic forms is fixed by the action, one may focus on that particular structure and apply quantisation with respect to it, an approach that was carried out by Andersen-Gukov-Pei~\cite{AGP16} in the case of the Hitchin moduli space.
Nonetheless, one may still wish to obtain a version of quantisation with respect to other Kähler forms, or to the hyperkähler structure as a whole.
In addition, the induced action on the sphere is in many cases transitive, suggesting again that a more ``global'' approach should be taken in that situation.

The latter is precisely the setup that we are going to address in this work.
Namely, we shall consider a hyperkähler manifold $M$ acted on by a compact Lie group $G$ by hyperkähler rotations and assume that the induced action on $\mathbb{C}P^{1}$ is transitive.
We will also assume given a smooth family of pre-quantum line bundles on $M$ parametrised by $\mathbb{C}P^{1}$ with a lifted equivariant $G$-action.
Carrying out geometric quantisation for each individual symplectic form will give rise to a family of Hilbert spaces, typically of infinite dimension.
We will then attempt to use representation theory to ``break down'' these spaces into finite-dimensional components and assemble each family into a vector bundle over $\mathbb{C}P^{1}$.
We study these objects explicitly and show that their structure is determined by the combinatorics of irreducible sub-representations in the Hilbert spaces.
In particular, we construct natural connections on these bundles and explicitly characterise their curvatures.
While the resulting connection on the overall family fails to be projectively flat, we notice that it defines a holomorphic structure on it.
Based on this, we propose a definition of the overall quantum Hilbert space as the super-cohomology of this object, thus obtaining a natural $G$-representation as a space of sections of a bundle over $\mathbb{C}P^{1}$.

\subsection{Description of the main construction}

Let us expand and detail the brief description sketched above.
Suppose a hyperkähler manifold $M$ is given, $G$ a compact connected Lie group acting on it by hyperkähler rotations.
By this we mean that $G$ acts on $M$ by isometries which permute the Kähler structures on $M$; we will additionally require that the induced action be transitive.
Since Kähler forms are parametrised by $\mathbb{C}P^{1}$, this corresponds to a surjective map $G \to \SO(3)$.
As we shall see, this implies that $G$ is covered by a product $\Sp(1) \times G_{0}$, with $\Sp(1)$ acting on $\mathbb{C}P^{1}$ in the usual way and $G_{0}$ fixing all Kähler structures.

Since no preferred symplectic form is given on $M$ it makes little sense to talk about a pre-quantum line bundle over the hyperkähler manifold.
Instead, we will assume given a Hermitian line bundle $(L,h)$ on $M$ and a pre-quantum connection $\nabla_{q}$ for each symplectic form $\omega_{q}$, depending smoothly on $q \in \mathbb{C}P^{1}$ in an appropriate sense.
We will further assume that the $G$-action lifts to $L$ permuting the connections equivariantly.

If $M_{q}$ denotes the Kähler manifold corresponding to $q \in \mathbb{C}P^{1}$, its geometric quantisation consists of the Hilbert space $\mathcal{H}_{q}$ of $\LL$ holomorphic sections of the corresponding pre-quantum line bundle.
The embedding into $\mathbb{C}P^{1} \times \LL(M,L)$ defines a Hermitian structure on this family, viewed informally as a vector bundle over $\mathbb{C}P^{1}$, together with a compatible connection and $\Sp^{1}$-equivariant $G$-action.
We study this object by decomposing each fibre $\mathcal{H}_{q}$ into istotypical components as a representation of (appropriate subgroups of) $G_{q} \coloneqq \Stab_{G}(q)$.
By transitivity of the $\Sp(1)$-action, all these stabilisers are conjugated and the respective isotypical components identified, thus forming constant-rank sub-families of $\mathcal{H}$.
The following is the central theorem of our work.

\begin{thm}[Cf.~Thm~\ref{thm:CurvaturePF}]
	\label{thm:main}
	Suppose that
	\begin{itemize}
		\item $M$ is a hyperkähler manifold;
		\item $G$ is a connected compact Lie group acting on $M$ by fixing the metric and permuting the symplectic forms transitively;
		\item $L \to M$ is a Hermitian line bundle with a family of pre-quantum connections as in \textsection~\ref{sec:Sp1equi} and a $G$-action covering that on $M$;
		\item $\rho$ is an irreducible representation of $G_{q} \coloneqq \Stab_{G}(\omega_{q})$ for $\omega_{q}$ one of the symplectic forms on $M$;
		\item $\rho$ has finite multiplicity $m^{(\rho)}$ in the space $\mathcal{H}_{q}$ of $\LL$ holomorphic sections of $L \to M$ with respect to the structure associated to $\omega_{q}$.
	\end{itemize}
	For each other symplectic form $\omega_{q'}$, call $\mathcal{H}_{q'}^{(\rho)}$ the isotypical component in $\mathcal{H}_{q'}$ corresponding to $\rho$ under the identification $G_{q} \simeq G_{q'}$ by conjugation in $G$.
	Then the collection of spaces $\mathcal{H}^{(\rho)}$ has a canonical structure as a Hermitian vector bundle over $\mathbb{C}P^{1}$ with compatible connection.
	Moreover, for some $d$ completely determined by $\rho$ there exists an isomorphism
	\begin{equation}
		\mathcal{H}^{(\rho)} \simeq \bigl( \mathcal{L}^{d} \otimes \underline{V_{\rho}} \bigr)^{\oplus m^{\rho}}
	\end{equation}
	preserving the Hermitian structure and connection, where $\underline{V_{\rho}} = \mathbb{C}P^{1} \times V_{\rho}$ with the trivial connection and $\mathcal{L}$ is the standard degree-$1$ $\Sp(1)$-equivariant Hermitian line bundle with connection over $\mathbb{C}P^{1}$.
\end{thm}

The result implies that, informally speaking, the family $\mathcal{H}$ decomposes as a sum of vector bundles with connections, as long as the appropriate multiplicities are finite.
The holonomies on the various components may then be assembled together to form parallel transport operators on $\mathcal{H}$.
The theorem, however, also determines the curvature of the connection on each component, which is proportional to the degree $d$ as in the theorem.
Consequently, the parallel transport operators on $\mathcal{H}$ depend essentially on the choice of paths on the base and fail to unambiguously identify the different Hilbert spaces, even projectively.

Nonetheless, the components $\mathcal{H}^{(\rho)}$ may also be regarded as $G$-equivariant \emph{holomorphic} bundles over $\mathbb{C}P^{1}$.
We then obtain $G$-representations not as spaces of flat sections as customary, but as the cohomology of $\mathcal{H}^{(\rho)}$ as a super vector space.
The following then descends from Thm.~\ref{thm:main}.

\begin{thm}[Cf. \textsection~\ref{sec:super_spaces}]
	\label{thm:super_reps}
	In the setting of Thm.~\ref{thm:main}, call
	\begin{equation}
		H^{(\rho)} \coloneqq H^{*} \bigl( \mathbb{C}P^{1}, \mathcal{H}^{(\rho)} \bigr)
	\end{equation}
	as a super vector space.
	Then $H^{(\rho)}$ comes with a Hermitian structure and compatible $G$-action.
	Calling $d$ as in Thm.~\ref{thm:main}, it is a direct sum of $\ \abs{d+1} m^{(\rho)}$ copies of $V_{\rho}$, all in even (resp. odd) degree if $d \geq 0$ (resp. $d<0$).
	In particular, the completed orthogonal sum
	\begin{equation}
		H \coloneqq \overline{\bigoplus_{\rho} H^{(\rho)}} \, ,
	\end{equation}
	with $\rho$ ranging over all the isomorphism classes of irreducible $G$-representations, defines a Hilbert space $G$-representation, and the above is the isotypical decomposition.
\end{thm}

This viewpoint also lends itself to an approach in terms of rank-generating series and localisation formul\ae, something which we address in \textsection~\ref{sec:gen_series}.

Again, this space $H$ may informally be thought of as the cohomology of the sum of all the $\mathcal{H}^{(\rho)}$'s, regarded now as a \emph{holomorphic} vector bundle over $\mathbb{C}P^{1}$.
It is interesting to note how this is reminiscent of the description of $M$ in terms of its twistor space, a holomorphic fibration $Z \to \mathbb{C}P^{1}$ (plus additional holomorphic data).
It would be an interesting problem to investigate whether our setup can be obtained in terms of twistor data by purely holomorphic constructions, something which we would like to address in a separate work.

The most crucial assumptions in our construction, besides the surjectivity of $G \to \SO(3)$, is the finite-dimensionality of the isotypical components in $\mathcal{H}_{q}$.
For that reason, we also investigate sufficient conditions to ensure it.
They can be summarised as follows.

\begin{thm}[Cf. Thmm.~\ref{thm:fdmm} and~\ref{thm:formulaHHprime}]
	Suppose one of the Kähler forms $\omega_{q}$ on $M$ is fixed, $S \subseteq \Stab_{G}(\omega_{q})$ a connected Lie subgroup, $\rho$ an isomorphism class of $S$-representations.
	Then each of the following is a sufficient condition for the corresponding isotypical component in $\mathcal{H}_{q}$ as an $S$-representation to have finite dimension:
	\begin{itemize}
		\item The Kostant moment map for $S$ is proper on $M_{q}$, and its action extends holomorphically to the complexification of $S$;
		\item $S$ is a torus, $M_{q}$ has the structure of an affine scheme or Stein space, and the Kostant moment map for $S$ is proper;
		\item $S$ is a torus, $M_{q}$ has the structure of an affine scheme or Stein space, and \hbox{$M \sslash_{w} S$}, $w$ the weight of $\rho$, admits a compactification with rational singularities and boundary of codimension greater than or equal to $2$.
	\end{itemize}
\end{thm}

A further way to ensure finite dimensionality can be found in the discussion of meromorphic torus actions in~\cite{Wu}.

\subsection{Applications and further directions}

In \S~\ref{sec:examples} we showcase applications of the main construction.
The first one is a hyperk\"ahler vector space $V$ of real dimension $4n$, with $n \in \mathbb{Z}_{\geq 1}$. 
In this case 
\begin{equation}
    \Hk(V) \simeq \Sp(n) \cdot \Sp(1) \, ,
\end{equation}
identifying $V \simeq \mathbb{H}^n$ (cf. Rem.~\ref{rem:notation}).
Indeed under this isomorphism $\Sp(1)$ acts on $V$ via right multiplication of unit-norm imaginary quaternions, and commutes with the natural $\Sp(n)$-action.
Further the norm associated to the hyperk\"ahler metric provides a hyperk\"ahler potential and we can apply the abstract construction (see Thmm.~\ref{thm:CurvHKV} and ~\ref{thm:HKvectMT}).

\vspace{5pt}

Importantly, there are many more examples of (nonflat) $\Sp(1)$-symmetric hyperk\"ahler manifolds. 
These include moduli spaces of magnetic monopoles on $\mathbb{R}^3$ by the work of Atiyah--Hitchin and Taubes~\cite{AH} or equivalently, by the work of Donaldson~\cite{Don84b}, the moduli spaces of based rational maps from $\mathbb{C}P^1$ to itself; moduli spaces of framed $\SU(r)$-instantons on $\mathbb{R}^4$, by the work of Maciocia~\cite{Maciocia}; the hyperk\"ahler structure on nilpotent orbits, by the work of Kronheimer~\cite{Kronheimer}, and more generally the hyperk\"ahler Swann bundle over any quaternionic K\"ahler manifold~\cite{Swann}. 
In four dimensions a complete classification of $\Sp(1)$-symmetric hyperk\"ahler manifolds is given (up to finite covers) by the work of Gibbons--Pope~\cite{S21} and by Atiyah--Hitchin~\cite{AH}.
The three examples are the flat metric on $\mathbb{H}$, the Taub--NUT metric, and the hyperk\"ahler metric on the moduli space of charge-2 monopoles, i.e. the Atiyah--Hitchin manifold. 

We establish in \S~\ref{sec:4DExample} and \ref{sec:Instantons} that the theorems~\ref{thm:main}, \ref{thm:super_reps} and~\ref{thm:formulaHHprime} (or slight modifications thereof) apply to some of these examples, producing a quantisation and corresponding irreducible unitary (super) representations of distinguished groups of hyperk\"ahler isometries.


\section*{Acknowledgements}

The authors would like to thank a number of people whose feedback improved the results of this work, including Nigel Hitchin, Eckhard Meinrenken, Steve Rayan, Andrew Swann, and the graduate students from the gLab at the University of Toronto.
Special thanks go to Maxence Mayrand, whose input was of crucial value.

All authors thank the former Centre for the Quantum Geometry of Moduli Spaces (QGM) at the Aarhus University, and the new Centre for Quantum Mathematics (QM) at the University of Southern Denmark, for hospitality and support during various phases of this project.

The first-named author was partially supported by the Danish National Science Foundation Center of Excellence grant, QGM DNRF95, and by the ERC-SyG project, Recursive and Exact New Quantum Theory (ReNewQuantum) which received funding from the European Research Council (ERC) under the European Union’s Horizon 2020 research and innovation programme under grant agreement No 810573.

The second-named author thanks for their support the University of Toronto and the University of Saskatchewan, as well as the Pacific Institute for the Mathematical Sciences (PIMS) and the Centre for Quantum Topology and its Applications (quanTA).

The research of the third-named author was supported by the National Center of Competence in Research (NCCR) SwissMAP, of the Swiss National Science Foundation, and by the Deutsche Forschungsgemeinschaft (DFG, German Research Foundation) under Germany’s Excellence Strategy-GZ 2047/1, Projekt-ID 390685813; their work is now funded by the European Union under the grant agreement n. 101108575 (HORIZON-MSCA project~\href{https://ec.europa.eu/info/funding-tenders/opportunities/portal/screen/how-to-participate/org-details/999999999/project/101108575/program/43108390/details}{QuantMod}).

Additionally, the authors wish to thank the anonymous referees for carefully reading a previous draft of this paper, and for their useful feedback and constructive comments.

\section{Abstract \texorpdfstring{$\Sp(1)$}{Sp(1)}-symmetric hyperk\"ahler quantisation}
\label{sec:setup}

\subsection{Hyperk\"ahler manifolds and their symmetry groups}

Let $n$ be a positive integer and $M$ a smooth manifold of dimension 4$n$.

\begin{defn}
A hyperk\"ahler structure on $M$ is a Riemannian metric $g$ and an ordered triple $(I,J,K)$ of covariant constant orthogonal automorphism of $\T M$ satisfying the quaternionic identities $I^2 = J^2 = K^2 = IJK = -\Id_{\T M}$.
\end{defn}

It follows that $I,J,K$ are $g$-skew-symmetric global sections of $\End( \operatorname{T}M) \to M$, and we denote $\SpanIJKM$ the three-dimensional real Lie algebra they span.

The hyperk\"ahler 2-sphere of complex structures of $(M,g,I,J,K)$ is 
\begin{equation}
    \label{eq:hyperkaehler_2_sphere}
    \mathbb{S}_{IJK} \coloneqq \Set{I_q = aI + bJ + cK | q = (a,b,c) \in \mathbb{R}^3, \, a^2+b^2+c^2 = 1} \subseteq \SpanIJKM \, .
\end{equation}

As customary, the structure on $M$ identifies~\eqref{eq:hyperkaehler_2_sphere} with the 2-sphere of unit-norm imaginary quaternions, i.e. to $\mathbb{C}P^1$ as a K\"ahler manifold.
In particular for $q \in \mathbb{C}P^1$ there is a (real) symplectic form on $M$ defined by 
\begin{equation}
\label{eq:kaehler_forms}
    \omega_q(v,w) \coloneqq g\big(I_q v,w\big), \qquad \text{for } v,w \in \T M \, .
\end{equation} 
The triple $M_q \coloneqq (M,I_q,\omega_q)$ is a K\"ahler manifold, and for further use we call $\mu_q = \dif \vol \in \Omega^{\topp}(M)$ the Liouville volume form---independent of $q \in \mathbb{C}P^1$ as it agrees with the Riemannian volume form of $(M,g)$.

\begin{rem}
	The above data can be encoded in a fibration $\pi_{\mathbb{C}P^1} \colon Z \to \mathbb{C}P^1$ of K\"ahler manifolds over the Riemann sphere, the \emph{twistor space} of $(M,g,I,J,K)$.
	Clearly this family comes with a natural global trivialisation $Z \simeq M \times \mathbb{C}P^1$ as a smooth fibre bundle, but not as fibre bundle with symplectic or complex fibres.
	Nonetheless the natural complex structure on $Z$ makes $Z \to \mathbb{C}P^1$ into a \emph{holomorphic} fibre bundle~\cite[pp.141-142]{Hit92}.
\end{rem}

Now consider the group $\mSp = \Sp(M,g,I,J,K) \subseteq \Iso(M,g)$ of Riemannian isometries of $(M,g)$ preserving the K\"ahler forms $\omega_q$ (or equivalently the complex structures $I_q$) simultaneously for all $q \in \mathbb{C}P^1$.
This group is often referred to as the \emph{hyper-unitary} group.
Denoting $\Aut_{0} (Z)$ the group of holomorphic automorphisms of $Z \to \mathbb{C}P^1$ over the identity, there is a natural morphism
\begin{equation}
\label{eq:Hk_act_Z_0}
	\mSp \longrightarrow \Aut_{0} (Z),
\end{equation}
given by the fibrewise action of $\mSp$.

We shall consider a group of isometries that preserve the hyperk\"ahler structure in a looser sense, relaxing the condition that differentials should commute with $I$, $J$, and $K$ individually.

\begin{defn}
    Let $\mHK \subseteq \Iso(M,g)$ be the subgroup stabilising the Lie algebra $\SpanIJKM$:
    \begin{equation}
        \mHK = \text{Hk}(M,g,I,J,K) \coloneqq \Set{ \varphi \in \Iso(M,g) | \Ad_{\dif \varphi} \bigl( \SpanIJKM \bigr) = \SpanIJKM } \, .
    \end{equation}
\end{defn}

Hence $\mHK$ acts on $\SpanIJKM$, and $\mSp \subseteq \mHK$ is the kernel of this action.
Moreover the adjoint action $\Ad$ on $\SpanIJKM \simeq \mathbb{R}^3$ is by positive isometries for the standard Euclidean structure, resulting in a group morphism 
\begin{equation}
    \Ad \colon \mHK \longrightarrow \SO(3) \, , \qquad \Ad \colon \varphi \longmapsto \Ad_{\dif\varphi} \, , 
\end{equation}
and an action on the hyperk\"ahler 2-sphere~\eqref{eq:hyperkaehler_2_sphere}---simply denoted $q \mapsto \varphi.q$.
The combination of the actions of $\Hk(M)$ on $\mathbb{C}P^{1}$ and $M$ itself naturally extends~\eqref{eq:Hk_act_Z_0} to a map
\begin{equation}
	\label{eq:Hk_act_Z}
	\Hk(M) \to \Aut(Z) \, ,
\end{equation}
where $\Aut(Z)$ denotes the full group of biholomorphisms of $Z$ compatible with the fibration map and covering arbitrary Kähler automorphisms of $\mathbb{C}P^{1}$.

Suppose now given a connected compact Lie group $G$, and a $G$-action
\begin{equation}
	\label{eq:G_act_M}
	\rho \colon G \to \Hk(M)
\end{equation}
on $M$ by transformations in $\Hk(M)$.
We will sometimes denote
\begin{equation}
	\label{eq:G_act_Z}
	\rho^{Z} \colon G \to \Aut(Z)
\end{equation}
the composition of $\rho$ with~\eqref{eq:Hk_act_Z}; where unambiguous, we will often denote the $G$-action simply by $(\rho(g))(p) = gp$, and similarly for $\rho^{Z}$.
As in the introduction, we require that the induced $G$-action on $\mathbb{C}P^{1}$ be transitive, or equivalently that the corresponding map $G \to \SO(3)$ be surjective.
The kernel $G_{0}$ of this action is then also a compact Lie group, and by construction it acts on $M$ by transformations in $\mSp$.

\begin{lem}
	\label{lem:Sp1action}
	The induced $G$-action on $\mathbb{C}P^{1}$ factors through a morphism
	\begin{equation}
		\sigma \colon \Sp(1) \to G
	\end{equation}
	from the universal cover $\Sp(1)$ of $\SO(3)$.
	This, moreover, arises from a covering map $G_{0} \times \Sp(1) \to G$.
\end{lem}

\begin{proof}
	By compactness, the Lie algebra $\mathfrak{g} \coloneqq \operatorname{Lie}(G)$ admits a non-degenerate invariant pairing.
	Chosen such a pairing, the orthogonal complement of $\mathfrak{g}_{0} \coloneqq \operatorname{Lie}(G_{0})$ is a sub-algebra which maps isomorphically to $\mathfrak{so}(3) \simeq \mathfrak{su}(2)$.
	This induces a section $\mathfrak{su}(2) \to \mathfrak{g}$ which integrates to the desired map $\sigma$.
	In fact, since $\mathfrak{g}_{0}$ and $\mathfrak{g}_{0}^{\perp}$ commute with each other, the splitting $\mathfrak{g} \simeq \mathfrak{g}_{0} \oplus \mathfrak{g}_{0}^{\perp}$ is an isomorphism of Lie algebras.
    In particular, every element of $G_{0}$ commutes with $\sigma(\Sp(1))$, resulting in a map $G_{0} \times \Sp(1)$ as claimed.
\end{proof}

In other words, for $G$ as above, a $G$-action by transitive hyperkähler rotations always comes from an $\Sp(1)$-action, and up to covers it splits as the product with an action by $\mSp$.
Henceforth we shall assume fixed a $\sigma$ as above.

\subsection{Pre-quantum data}
\label{sec:Sp1equi}

As already noted, a notion of pre-quantum line bundle on $M$ is ill-posed, since a hyperkähler manifold comes with a continuous family of incompatible pre-quantum conditions.
Instead, we will assume given a Hermitian line bundle $(L,h)$ on $M$ together with a smooth family of compatible connections $\nabla_{q}$, $q\in\mathbb{C}P^{1}$, each with curvature $F_{q}=-i\omega_{q}$.
The smoothness in $q$ may be expressed by the condition that, if a section $\psi$ of $\pi_{M}^{*} L \to Z$ is smooth, then so is the family $\nabla_{q} \psi\rvert_{q}$, as a section of $\pi_{M}^{*} (L \otimes T^{*} M)$.
Equivalently, for every local trivialisation of $L$, the induced connection potentials should depend smoothly on $q \in \mathbb{C}P^{1}$.
Together with the trivial derivative along the directions of $\mathbb{C}P^{1}$ in $Z \simeq M \times \mathbb{C}P^{1}$, these $\nabla_{q}$'s assemble to form a connection on $\pi^{*} L \to Z$.
Additionally, we will need to require that $L$ be equipped with a Hermitian $G$-action
\begin{equation}
	\label{eq:G_act_L}
	\rho^{L} \colon G \to \Aut(L,h)
\end{equation}
which lifts the one on $M$ and permutes the connections equivariantly.

In practice, the $G$-action may not always come with preferred pre-quantum data as above.
We shall now investigate criteria to determine whether such data exists for a given action.

A necessary condition for the existence of a pre-quantum line bundle on a symplectic manifold is that the symplectic form represent an integral class in cohomology.
Conversely, in that case a pre-quantum line bundle can be constructed by a diagram chasing procedure on the \v{C}ech--de Rham complex~\cite{Woo80}.

In our situation, we will need to require that
\begin{equation}
    \bigl[ \omega_q \bigr] \in H^2(M,\mathbb{Z}) \, \text{ for all } q \in \mathbb{C}P^1 \, .
\end{equation}
In fact, if the condition holds for at least one $q$, then by the $\Sp(1)$-action it does for all $q$, and it then follows by continuity that $[\omega_{q}]$ is independent of $q$.
The diagram chasing procedure mentioned above may then be carried out with differential forms on $M$ depending smoothly on $q$.
Hence a family of pre-quantum line bundles exists if and only if $[\omega_{q}]$ is integral for some $q$.

Suppose such a family is fixed, with underlying Hermitian line bundle $(L,h)$, and call $L_{q} \coloneqq (L,h,\nabla_{q})$ for each $q$.
For every $g \in G$, the structure of $g^{*} L_{g.q} \otimes L_{q}^{-1}$ defines a family of \emph{flat} Hermitian connections.
Since such objects are classified up to isomorphism by $\Gamma \coloneqq H^1 \bigl( M,\U(1) \bigr)$, this defines a map
\begin{equation}
	\label{eq:u-cocycle}
    u \colon G \longrightarrow C^\infty (\mathbb{C}P^1, \Gamma) \, , \qquad u \colon g \longmapsto \Bigl( q \mapsto \bigl[ g^*L_{g.q} \otimes L_{q}^{-1} \bigr] \Bigr) \, .
\end{equation}
Viewing the abelian group $\Gamma' \coloneqq C^{\infty} (\mathbb{C}P^{1}, \Gamma)$ as a $G$-module under the pull-back action, $u$ defines a cocycle in $C^{1} (G, \Gamma')$.

\begin{lem}
	Suppose $(L,h)$ is a Hermitian line bundle over $M$ with a family of pre-quantum connections $\nabla_{q}$ smoothly parametrised by $\mathbb{C}P^{1}$.
	The cohomology class of the coycle $u$ from~\eqref{eq:u-cocycle} vanishes in $H^{1} (G, \Gamma')$ if and only if there exist a Hermitian line bundle $B$ and a family of Hermitian flat connections $\nabla^{B}_{q}$ smoothly parametrised by $\mathbb{C}P^{1}$ such that, for all $q \in \mathbb{C}P^{1}$ and $g \in G$,
	\begin{equation}
		\label{eq:twist}
		g^{*} \bigl(L_{g.q} \otimes B_{g.q}^{-1} \bigr) \simeq L_{q} \otimes B_{q}^{-1}
	\end{equation}
	as Hermitian bundles with connection, where $B_{q} \coloneqq (B, \nabla^{B}_{q})$.
\end{lem}

\begin{proof}
	Suppose such a family exists.
	Then~\eqref{eq:twist} is equivalent to
	\begin{equation}
		g^{*} L_{g.q} \otimes L_{q}^{-1} \simeq g^{*} B_{g.q} \otimes B_{q}^{-1} \, ,
	\end{equation}
	i.e. $u = \delta \Lambda$ for $\Lambda(q) \coloneqq [B_{q}] \in \Gamma$, and therefore $[u] = 0$.
	
	Conversely, suppose that $u = \delta \Lambda$ for some $\Lambda \in \Gamma'$.
	It follows from the exact sequence
	\begin{equation}
		H^{1}(M, \mathbb{R}) \to H^{1} \bigl(M,\U(1)\bigr) \to H^{2}(M,\mathbb{Z}) \to H^{2}(M,\mathbb{R})
	\end{equation}
	that the components of $\Gamma = H^{1}\bigl(M,\U(1)\bigr)$ are labelled by the torsion of $H^{2}(M,\mathbb{Z})$, while the identity component is covered by $H^{1} (M,\mathbb{R})$.
	Since $\Lambda \colon \mathbb{C}P^{1} \to \Gamma$ is a continuous map, we may fix some $\Lambda_{0} \in \Gamma$ so that $\Lambda-\Lambda_{0}$ takes values in the identity component.
	Since $\mathbb{C}P^{1}$ is simply connected, this lifts to a map $\widetilde{\Lambda} \colon \mathbb{C}P^{1} \to H^{1} (M,\mathbb{R})$.
	Choose a collection of $1$-forms $\alpha_{1}, \ldots, \alpha_{n}$ on $M$ whose de Rham cohomology classes form a basis of $H^{1} (M, \mathbb{R})$.
	Expressing $\widetilde{\Lambda}$ as
	\begin{equation}
		\widetilde{\Lambda} (q) = \sum_{i = 1}^{n} c_{i} (q) [\alpha_{i}] \, ,
	\end{equation}
	we see that each $c_{i}$ is a smooth function of $q$ and therefore $\alpha = \sum_{i = 1}^{n} c_{i} \alpha_{i}$ is a smooth family of $1$-forms on $M$ parametrised by $\mathbb{C}P^{1}$.
	Choosing a representative $(B,\nabla^{B}_{0})$ of $\Lambda_{0} \in H^{1} (M, \U(1))$ and setting $\nabla^{B}_{q} \coloneqq \nabla^{B}_{0} + \alpha(q)$, it follows by construction that
	\begin{equation}
		\Lambda(q) = \bigl[(B, \nabla_{q}^{B}) \bigr] \, .
	\end{equation}
	Expanding and manipulating the condition $u = \delta \Lambda$ leads to~\eqref{eq:twist}, as desired.
\end{proof}

The lemma shows that, if $[u]=0$, then $L$ may be replaced by a new family of pre-quantum line bundles on which the action of every element of $G$ admits an equivariant lift.
In that case, the action of the group
\begin{equation}
	G' \coloneqq \Set{ \varphi \in \Aut(L,h) | \varphi \text{ covers some $g \in G$}}
\end{equation}
covers that of $G$ on $M$ surjectively while permuting the connections equivariantly.
Notice moreover that $G'$ is also compact and connected, being a central extension of the image of $G$ in $\Hk(M)$, so the discussion from the previous section also applies to it.
In particular, by Lem.~\ref{lem:Sp1action}, there exists a map $\sigma^{L} \colon \Sp(1) \to G'$ lifting the $\Sp(1)$-action on $M$.
Even though there may not be a lifted $G$-action on $L$, we obtain one by replacing the group with $G'$, which does not essentially change the action on $M$.

The simplest vanishing $[u] = 0$ is obtained if $\Gamma$ is trivial (which we will see in some examples) or by the existence of a hyperk\"ahler potential (which we discuss in \S~\ref{sec:HKpotSec}).

Up to the necessary replacements, in what follows we thus assume to have fixed a family of pre-quantum connections with an equivariant action $\rho^{L} \colon G \to L$.

\subsection{Geometric quantisation}

Following the prescription of geometric quantisation, for $q \in \mathbb{C}P^1$ consider the separable Hilbert space
\begin{equation}
\label{eq:quantum_space}
    \mathcal{H}_q \coloneqq \Set{ \psi \in H^0(M_q,L_{q}) | \int_M h(\psi,\psi) \dif \vol < \infty}  \subseteq \LL(M,L) \, ,
\end{equation} 
using the holomorphic structure $\overline{\partial}_q = \nabla_q^{0,1}$ and the standard $\LL$ Hermitian product:
\begin{equation}
\label{eq:hermitian_product}
    \langle \psi \mid \psi' \rangle \coloneqq \int_M h(\psi,\psi') \dif \vol \, , \qquad \text{for } \psi,\psi' \in \mathcal{H}_q \, .
\end{equation}
Let us denote by $\mathcal{H}$ the family of Hilbert spaces thus defined over $\mathbb{C}P^1$.

By construction there are unitary isomorphisms
\begin{equation}
    \rho^{\mathcal{H}}_{g} \colon \mathcal{H}_q \longrightarrow \mathcal{H}_{g.q} \qquad \text{for } q \in \mathbb{C}P^1 \text{ and } g \in G \, ,
\end{equation}
explicitly given by
\begin{equation}
\label{eq:pull_back_action}
	\bigl( \rho^{\mathcal{H}}_{g}\psi \bigr) (m) \coloneqq \rho_{g}^L \Bigl( \psi \bigl(\rho^{Z}_{g^{-1}} m \bigr) \Bigr) \qquad \text{for } m \in M \, .
\end{equation}

\subsection{Decomposition of \texorpdfstring{$\mathcal{H}_q$}{H-q}}
\label{sec:The_grading}

We will now consider the decompositions of the spaces~\eqref{eq:quantum_space} induced by viewing them as representations under the action~\eqref{eq:pull_back_action}.
For a given $q \in \mathbb{C}P^{1}$, restricting $\rho^{\mathcal{H}}$ to
\begin{equation}
	G_{q} \coloneqq \Stab_{G}(q) \, ,
\end{equation}
defines a group action on $\mathcal{H}_{q}$ by unitary operators, i.e. a Hilbert space representation.
By the Peter--Weyl theorem~\cite[Thm.~1.12]{Knapp}, $\mathcal{H}_{q}$ decomposes a completed orthogonal sum of irreducible components.
Similarly, calling
\begin{equation}
	T_{q} \coloneqq \Stab_{\Sp(1)} (q)
\end{equation}
the maximal torus in $\Sp(1)$ fixing $q$, its action on $\mathcal{H}_{q}$ gives a decomposition
\begin{equation}
\label{eq:decomp}
	\mathcal{H}_q = \overline{\bigoplus_{d \in \mathbb{Z}} \mathcal{H}_q^{(d)}}  \, ,
\end{equation}
where $\mathcal{H}^{(d)}_q \subseteq \mathcal{H}_q$ is the isotypical component corresponding to the character 
\begin{equation}
    T_q \simeq \U(1) \longrightarrow \mathbb{C}^{\times} \, , \quad z \longmapsto z^d \, ,
\end{equation}
under the natural identification with the standard torus $\operatorname{U}(1) \subseteq \mathbb{C}^{\times}$.
Since $T_{q}$ commutes with $G_{0} \subseteq G_{q}$, each component $\mathcal{H}_{q}^{(d)}$ is a representation of $G_{0}$.
Therefore, we obtain a refinement of the decomposition above as
\begin{equation}
\label{eq:decompSp}
    \mathcal{H}_q^{(d)} =  \overline{\bigoplus_{\lambda \in \Lambda} \mathcal{H}_{q,\lambda}^{(d)}} \, ,
\end{equation}
where $\Lambda$ denotes the set of (analytically) integral weights of $G_{0}$ and $\mathcal{H}_{q,\lambda}^{(d)}$ is the isotypical component in $\mathcal{H}_{q}^{(d)}$ of maximal weight $\lambda$.
In what follows we shall often call $\Lambda^{(d)} \subseteq \Lambda$ the subset of ``active'' representations.%
\footnote{The subset $\Lambda^{(d)}$ is independent of $q \in \mathbb{C}P^1$ since the $\mpSpo$-modules $\mathcal{H}_q^{(d)}$ are isomorphic under the $\Sp(1)$-action.}

\begin{rem}
	\label{rem:rho=d_lambda}
	It is not difficult to see, given that $G_{q}$ is covered by $T_{q} \times G_{0}$, that every irreducible representation of $G_{q}$ has a single weight for $T_{q}$ and is also irreducible for $G_{0}$.
	Therefore, every irreducible $G_{q}$-representation induces a pair $(d,\lambda)$ of weights for $T_{q}$ and $G_{0}$, by which the representation itself is unambiguously determined.
	In particular, the decomposition~\eqref{eq:decompSp} is equivalent to the one into isotypical components under $G_{q}$.
\end{rem}

In a similar way we may also consider a maximal torus $T \subseteq G_{0}$ and find
\begin{equation}
\label{eq:decompMT}
	\mathcal{H}^{(d)}_q = \overline{\bigoplus_{a \in T^\vee} \mathcal{H}^{(d)}_{a,q}} \, ,
\end{equation}
where $\mathcal{H}^{(d)}_{a,q} \subseteq \mathcal{H}^{(d)}_q$ is the isotypical component of the character $a \colon T \to \mathbb{C}^{\times}$.
Again, the decomposition above is equivalent to the one we would obtain by considering the action of the maximal torus $T_{q}' \coloneqq T_{q} \cdot T \subseteq G_{q}$ on $\mathcal{H}_{q}$.

We denote $\mathcal{H}^{(d)}$, $\mathcal{H}_{a}^{(d)}$ and $\mathcal{H}_{\lambda}^{(d)}$ the families of Hilbert spaces thus defined over $\mathbb{C}P^1$, so that we have $\LL$-completed orthogonal direct sums
\begin{equation}
    \mathcal{H} = \overline{\bigoplus_{d \in \mathbb{Z}} \mathcal{H}^{(d)}} \, , \qquad \overline{\bigoplus_{\lambda \in \Lambda} \mathcal{H}_{\lambda}^{(d)}} = \mathcal{H}^{(d)} = \overline{\bigoplus_{a \in T^{\vee}} \mathcal{H}_{a}^{(d)}} \, .
\end{equation}

\subsection{Structure of \texorpdfstring{$\mathcal{H}^{(d)}_\lambda$}{H-lambda}}
\label{sec:untwisted_connection}

Our main assumption is the following, unless otherwise stated:

\begin{center}
    \framebox[1.1\width]{
    The spaces $\mathcal{H}_{q,\lambda}^{(d)}$ are finite-dimensional.}
\end{center}

The goal of this section is to make each family $\mathcal{H}_{\lambda}^{(d)}$, under the finite-dimensionality condition above, into a vector bundle, and equip it with a connection.
The simplest way to define a smooth structure on $\mathcal{H}_{\lambda}^{(d)}$ is to assume it is a Banach sumbanifold of the product $\mathbb{C}P^{1} \times \LL(M,L)$, with $G$ acting smoothly on it.
We shall now investigate the structure induced by this assumption, and later show that the same data can be obtained canonically from the group structure.
    
Assume then (temporarily, cf. Rem.~\ref{rem:alternative_definition} and below) that the family of Hilbert spaces $\mathcal{H}^{(d)}_{\lambda}$ forms a \emph{smooth} Banach sub-manifold of the trivial Hilbert bundle $\underline{\LL(M,L)} \to \mathbb{C}P^1$.
We can then differentiate smooth local sections $\psi$ of $\mathcal{H}^{(d)}_{\lambda} \to \mathbb{C}P^1$, viewed as maps $\mathbb{C}P^{1} \to \LL(M,L)$, along tangent vectors on $\mathbb{C}P^1$.
Then, since $\mathcal{H}^{(d)}_{q,\lambda} \subseteq \LL (M,L)$ is a closed subspace, there are orthogonal projections
\begin{equation}
	\pi^{(d)}_{q,\lambda} \colon \LL (M ,L) \longrightarrow \mathcal{H}^{(d)}_{q,\lambda} \, .
\end{equation}

\begin{defn}
	\label{def:L^2_connection}
    For any tangent vector $X \in T_q \mathbb{C}P^1$ set 
    \begin{equation}
    	\nabla_{X}^{\mathcal{H}_\lambda^{(d)}} \psi \coloneqq \pi^{(d)}_{q,\lambda} \Bigl( X[\psi] \Bigr) \in \mathcal{H}^{(d)}_{q,\lambda} \, .
    \end{equation}
\end{defn}

\begin{rem}
\label{Rem21}
    The same definition (of the standard $\LL$-connection) can be given verbatim in the case where the families $\mathcal{H}^{(d)}_a \subseteq \mathcal{H}^{(d)}$ also constitute smooth submanifolds.
\end{rem}

\begin{rem}
	\label{rem:char_nabla}
	This covariant derivative is characterised by the property that
	\begin{equation}
		\Braket{ \nabla_{X}^{\mathcal{H}_{\lambda}^{d}} \psi | \psi'} = \Braket{X[\psi] | \psi'}
	\end{equation}
	for all $X, \psi, \psi'$ as appropriate.
\end{rem}

\begin{prop}
\label{thm:invariant_metric_connection}
	The covariant-derivative operators of Def.~\ref{def:L^2_connection} are compatible with $\rho^{\mathcal{H}}$ and the Hermitian structure on $\mathcal{H}_\lambda^{(d)} \to \mathbb{C}P^1$.
\end{prop}

\begin{proof}
	The operators $\nabla_{X}^{\mathcal{H}_{\lambda}^{(d)}}$ satisfy Leibnitz and preserve the Hermitian pairing by construction.
	We need only show that they are $\rho^{\mathcal{H}}$-equivariant.
	Given $g \in G$, $q \in \mathbb{C}P^{1}$, a section $\psi$ of $\mathcal{H}_{\lambda}^{(d)} \to \mathbb{C}P^1$, and a tangent vector $X \in \operatorname{T}_{q} \mathbb{C}P^1$, we have
	\begin{equation}
		X \bigl[ \rho_{g} \psi \bigr] = \rho_{g} \bigl( (g_{*}^{-1} X) [\psi] \bigr) \, ,
	\end{equation}
	where the superscripts in the actions were removed for convenience.
	Combining the above with a change of variables in~\eqref{eq:hermitian_product}, one sees that
	\begin{equation}
	    \Braket{ X \bigl[\rho_{g} \psi \bigr] | \psi'} = \Braket{ \bigl(g_{*}^{-1} X\bigr) [\psi] | \rho_{g^{-1}} \psi'}
	\end{equation}
	for all $\psi' \in \mathcal{H}_{g.q,\lambda}^{(d)}$.
	By Rem.~\ref{rem:char_nabla}, this shows that $\nabla_{X} \bigl( \rho_{g} \psi \bigr) = \rho_{g} \Bigl( \nabla_{g_{*}^{-1} X} [\psi] \Bigr)$.
\end{proof}
	
Recall now that for every integer $d$ there exists an $\Sp(1)$-equivariant Hermitian line bundle of each degree $d$ with compatible connection over $\mathbb{C}P^{1}$, unique up to isomorphism.
This can be characterised as the holomorphic line bundle $\mathcal{O}(d)$ together with the standard Hermitian metric and its corresponding Chern connection.
Alternatively, it can also be described as the quotient of an appropriate line bundle over $\Sp(1)$ under the identification $\mathbb{C}P^{1} \simeq \Sp(1) \slash \U(1)$.
More precisely, consider the $d$-th character $\chi^{(d)} \colon \mathfrak{u}(1) \to \mathbb{R}$ and its unique $\Ad_{\U(1)}$-invariant extension to $\mathfrak{sp}(1)$.
Calling $\alpha^{(d)}$ the corresponding left-invariant $1$-form on $\Sp(1)$, the connection $\dif + 2 \pi i \alpha^{(d)}$ on $\Sp(1) \times \mathbb{C}$ is then invariant under the actions
\begin{equation}
	A \cdot (x,z) \coloneqq (Ax,z)
	\qquad \text{and} \qquad
	(x,z) \cdot h \coloneqq (xh,h^{-d}z)
\end{equation}
for $A \in \Sp(1)$ and $h \in \U(1)$.
Furthermore, the right $\U(1)$-action is by construction horizontal for this connection.
Therefore, the latter descends to to a metric and $\Sp(1)$-equivariant connection on $\bigl(\Sp(1) \times \mathbb{C}\bigr) \slash \U(1) \to \Sp(1) \slash \U(1) \simeq \mathbb{C}P^{1}$.

Uniqueness can be established by noticing that the difference of two such line bundles comes with a connection whose curvature is $\Sp(1)$-invariant and vanishes in cohomology, and is therefore zero.
The space of flat sections is then a $1$-dimensional $\Sp(1)$-representation, so that choosing one unit element in this space gives an isomorphism of the line bundles intertwining the Hermitian structures and connections.

We will refer to this object as $\mathcal{L}^{d}$.

For each $d \in \mathbb{Z}$ and $\lambda \in \Lambda^{(d)}$, call $m_{\lambda}^{(d)}$ the multiplicity of $V_{\lambda}$ in $\mathcal{H}^{(d)}_{q}$.%
\footnote{The integer $m^{(d)}$ is independent of $q \in \mathbb{C}P^1$ (cf. the previous footnote).}
Finally, call $\underline{V_{\lambda}} \to \mathbb{C}P^{1}$ the trivial Hermitian bundle with fibre $V_{\lambda}$ with the trivial connection $\nabla^{\triv}$ and $\Sp(1)$ acting on it trivially on the fibres.

\begin{thm}
\label{thm:CurvaturePF}
	Suppose an integer $d$ and a dominant weight $\lambda$ of $G_{0}$ are fixed.
	Suppose that, for some $q \in \mathbb{C}P^{1}$, the corresponding isotypical component in the quantum Hilbert space $\mathcal{H}_{q}$ has finite multiplicity $m_{\lambda}^{(d)}$.
	Consider the collection $\mathcal{H}_{\lambda}^{(d)}$ of corresponding isotypical components, and suppose it forms a Banach submanifold of $\mathbb{C}P^{1} \times \LL(M,L)$ acted on smoothly by $G$.
	Then $\mathcal{H}_{\lambda}^{(d)}$ is a Hermitian vector bundle over $\mathbb{C}P^{1}$ and there exists a $G$-invariant isomorphism
	\begin{equation}
		\mathcal{H}_{\lambda}^{(d)} \simeq \bigl( \mathcal{L}^{d} \otimes \underline{V_{\lambda}} \bigr)^{\oplus m_{\lambda}^{(d)}}
	\end{equation}
	of Hermitian vector bundles which intertwines the covariant derivative operators $\nabla^{\mathcal{H}_{\lambda}^{(d)}}$ of Def.~\ref{def:L^2_connection} with the natural connection on the right-hand side.
\end{thm}

\begin{proof}
	Call for simplicity $m \coloneqq m_{\lambda}^{(d)}$, fix $q \in \mathbb{C}P^{1}$, and identify $T_{q}$ with $\U(1)$ by the orientation defined by $q$.
	Consider on $\Sp(1)$ the trivial vector bundle $\Sp(1) \times V_{\lambda}^{\oplus m}$ with the left and right actions
	\begin{equation}
		(A,g) \cdot (x,v) \coloneqq (Ax,gv)
		\qquad \text{and} \qquad
		(x,v) \cdot h = (xh,h^{-d}v)
	\end{equation}
	of $\Sp(1) \times G_{0}$ and $\U(1)$, respectively.
	Choose an isomorphism $\varphi \colon V_{\lambda}^{(\oplus m)} \to \mathcal{H}_{q,\lambda}^{(d)}$ as $G_{0}$-modules and define
	\begin{equation}
		\Phi \colon \Sp(1) \times V_{\lambda}^{\oplus m} \to \mathcal{H}_{\lambda}^{(d)} \, ,
		\qquad
		\Phi \colon (x,v) \mapsto \rho^{\mathcal{H}}(x) \bigl(\varphi(v)\bigr) \, .
	\end{equation}
	It is clear by construction that $\Phi$ is invariant under the right $\U(1)$-action and intertwines the $\Sp(1) \times G_{0}$-actions.
	It is also a surjective smooth map covering the projection $\pi \colon \Sp(1) \to \mathbb{C}P^{1}$, $\pi(x) \coloneqq xq$ and restricts fibre-wise to unitary isomorphisms.
	It follows that $\Phi$ is a submersion, and therefore the induced bijection
	\begin{equation}
		\label{eq:iso_below}
		\Bigl( \Sp(1) \times V_{\lambda}^{\oplus m} \Bigr) \Bigr \slash \U(1) \longrightarrow \mathcal{H}_{\lambda}^{(d)}
	\end{equation}
	is a diffeomorphism, thus showing that $\mathcal{H}_{\lambda}^{(d)}$ is a vector bundle as claimed.
	It then follows from Prop.~\ref{thm:invariant_metric_connection} that $\nabla^{\mathcal{H}_{\lambda}^{(d)}}$ is a Hermitian $G$-invariant connection.
		
	The map $\Phi$ may also be regarded as a unitary isomorphism
	\begin{equation}
		\label{eq:lifted_iso}
		\Sp(1) \times V_{\lambda}^{\oplus m} \simeq \pi^{*} \mathcal{H}_{\lambda}^{(d)} \, .
	\end{equation}
	Both sides come with $\Sp(1) \times G_{0}$- and $\U(1)$-invariant Hermitian connections, both making the right $\U(1)$-action horizontal.
	Such a connection, however, is uniquely characterised by these properties.
	Indeed, left $\Sp(1)$-invariance implies that such a connection is determined by the potential over any element of $\Sp(1)$.
	On the other hand, combining the left and right $\U(1)$-invariance shows that the operation of lifting elements of $\Tan_{\Id} \Sp(1) \simeq \mathfrak{sp}(1)$ horizontally is $\Ad_{\U(1)}$-equivariant.
	The condition that the right $\U(1)$-action be horizontal, moreover, determines the lifts of vectors in $\mathfrak{u}(1)$, and therefore of those in $\mathfrak{sp}(1)$ by $\Ad_{U}(1)$-invariance.
	We conclude that the isomorphism~\eqref{eq:lifted_iso} also identifies the connections on the two bundles, which is to say that the isomorphism~\eqref{eq:iso_below} is also horizontal.
	The left-hand side of~\eqref{eq:iso_below}, however, is clearly isomorphic to $\mathcal{L}^{d} \otimes \underline{V_{\lambda}}^{\oplus m}$.
	Finally, since the kernel of the covering map $\Sp(1) \times G_{0} \to G$ acts trivially on the right-hand side, it follows that the group action on the left-hand side descends to $G$.
\end{proof}

\begin{rem}
\label{rem:alternative_definition}
    Thm.~\ref{thm:CurvaturePF} yields an alternative definition of the bundles of isotypical components, without smoothness assumptions.
    Indeed, a map $\Phi$ constructed as above uniquely defines a smooth structure on $\mathcal{H}_{\lambda}^{(d)}$ making it a vector bundle with an isomorphism with $\mathcal{L}^{d} \otimes \underline{V_{\lambda}}^{m}$, and therefore inducing also a connection with the desired properties.
    Given that the only ambiguity in the construction of $\Phi$ lies in the choice of $\varphi$, any two such maps are related by pre-composition with a $G_{0}$-invariant automorphism of $V_{\lambda}^{\oplus m}$.
    Since this operation preserves the structure on $\Sp(1) \times V_{\lambda}^{\oplus m}$, the two choices induce the same data on $\mathcal{H}_{\lambda}^{(d)}$.

    This yields finite-rank smooth $G$-equivariant Hermitian vector bundles over the Riemann sphere, equipped with Hermitian connections, defined from the combinatorial data of the multiplicities of $\mathcal{H}_{q}$ as a representation, as long as the main assumption that the $\mathcal{H}_{q,\lambda}^{(d)}$'s be finite-dimensional is verified.
\end{rem}

Together with Rem.~\ref{rem:rho=d_lambda}, the content of this section proves Thm.~\ref{thm:main}.

\subsection{Quantum super Hilbert spaces and unitary representations}
\label{sec:super_spaces}

We now call $H^{(d)}_{\lambda}$ the super vector space obtained by taking the holomorphic cohomology of the bundles of isotypical components:
\begin{equation}
    H^{(d)}_{\lambda} \coloneqq H^* \bigl( \mathbb{C}P^1, \mathcal{H}^{d}_{\lambda} \bigr) \, .
\end{equation}
By Rem.~\ref{rem:rho=d_lambda}, the above is equivalent to the space we called $H^{(\rho)}$ in Thm.~\ref{thm:super_reps} in the introduction.
Since $\mathcal{H}_{\lambda}^{(d)}$ is Hermitian and $\mathbb{C}P^{1}$ is Kähler, the $\LL$-pairing on harmonic representatives gives each of the above has a natural Hermitian structure.

If $W^{(d)} =  W_+^{(d)} \oplus W_-^{(d)}$ is the unitary super $\Sp(1)$-representation defined by
\begin{equation}
    W_+^{(d)}  \coloneqq H^0 \bigl( \mathbb{C}P^1, \mathcal{L}^{d} \bigr) \, , \qquad W_-^{(d)}
	\coloneqq H^1 \bigl( \mathbb{C}P^1, \mathcal{L}^{d} \bigr) \, ,
\end{equation}
then $H^{(d)}_{\lambda} \simeq W^{(d)} \otimes V_\lambda^{\oplus m_{\lambda}^{(d)}}$ as super $G$-representations, regarding $V_{\lambda}$ as of pure even degree.
Moreover, $\dim W_{+}^{(d)}$ is equal to $d+1$ if $d \geq 0$ and $0$ otherwise, while similarly $\dim W_{-}^{(d)}$ vanishes for $d \geq 0$ and is equal to $-d-1$ otherwise.

Finally consider the nested $\LL$-completed orthogonal direct sums
\begin{equation}
\label{eq:direct_sum_superspaces}
    H \coloneqq \overline{\bigoplus_{d \in \mathbb{Z}} H^{(d)}} \, , \qquad H^{(d)} \coloneqq \overline{\bigoplus_{\lambda \in \Lambda^{(d)}} H_\lambda^{(d)}} \, .
\end{equation}

The discussion of this section finally provides a model for the quantisation of $(M,g,I,J,K)$ and its $G$-action as a $G$-representation, and proves Thm.~\ref{thm:super_reps} from the introduction.

\subsection{Finite-rank conditions}
\label{sec:finite_rank_conditions}

We shall now consider conditions which entail finite-dimensionality for the isotypical components of \S~\ref{sec:The_grading}.

In general, if $K$ is a compact Lie group with Lie algebra $\mathfrak{k} = \Lie(K)$, acting on a Kähler manifold $X$ with a lifted $K$-action on a pre-quantum line bundle $(L, \nabla)$, there is a natural moment map $\mu \colon X \to \mathfrak{k}^{\vee}$ defined by Kostant's formula
\begin{equation}
\label{eq:Kostant_moment}
    2 \pi i \braket{\mu, \xi} \frac{\partial}{\partial \theta} = \xi_{X}^{H} - \xi_{L}
\end{equation}
for every $\xi \in \mathfrak{k}$, where $\xi_{L}$ is the vector field corresponding to $\xi$ on $L$, $\xi_{X}^{H}$ the one on $X$ lifted horizontally, and $\frac{\partial}{\partial \theta}$ is the fibre-wise ``angular'' vector field.
In this setup, we will make use of the following version of the general principle that ``quantisation commutes with reduction''.

\begin{thm}[\cite{GS,Sja95}]
\label{thm:[Q,R]=0}
	In the setup above, if the $K$-action extends holomorphically to the complexified group $K^{\mathbb{C}}$, and if the moment map~\eqref{eq:Kostant_moment} is proper, then for every dominant weight $\gamma$ of $K$ there is an identification
	\begin{equation}
		\Hom_{K} \bigl( V_{\gamma}, H^{0}(X, L) \bigr) \simeq H^{0} \left( X_{\gamma}, L_{\gamma} \right) \, ,
	\end{equation}
	where $V_{\gamma}$ denotes a simple $K$-module of highest weight $\gamma$, $X_{\gamma} = X \sslash_{\gamma} K$ is the symplectic reduction of $X$ at level $\gamma$, and $L_{\gamma}$ is the induced ($V$-)bundle on $X_{\gamma}$.
\end{thm}

This result was first established by Guillemin and Sternberg~\cite{GS} in the case $X$ is compact, with additional regularity conditions, and then extended by Sjamaar~\cite{Sja95}.
The statement has been subsequently generalised in various works including those of Meinrenken~\cite{Mei96a,Mei96}, Meinrenken--Sjamaar~\cite{MS99}, Vergne~\cite{Vergne1,Vergne2}, Ma~\cite{Ma}, Ma--Zhang~\cite{MaZhang}, and Hochs--Song~\cite{HS}.

We emphasise that this formulation of ``quantisation commutes with reduction'' requires no assumptions on $\gamma$ being a regular value or the $K$-action being free on $\mu^{-1} (\gamma)$.
In the statement of Thm.~\ref{thm:[Q,R]=0}, $X_{\gamma}$ and $L_{\gamma}$ are regarded as a complex analytic space and a coherent sheaf, respectively.
See~\cite{Sja95} for further detail.

\vspace{5pt}

Returning to our setting, for any fixed $q$ and Lie subgroup $S \subset G_{q}$ we have a moment map
\begin{equation}
	\mu_{S} \colon M \to \Lie(S)^{\vee}
\end{equation}
given by Kostant's formula.
Then Thm.~\ref{thm:[Q,R]=0} yields the following.

\begin{thm}
\label{thm:fdmm}
	Suppose that $q \in \mathbb{C}P^{1}$ and $S \subseteq G_{q}$ is a (connected) Lie subgroup, and call $\mu_{S}$ the Kostant moment map of the $S$-action on $M_{q}$.
	Assume moreover that $\mu_{S}$ is proper, and that the $S$-action has a holomorphic extension to the complexified group $S^{\mathbb{C}}$ on $M_{q}$.
    Then every isotypical component in $\mathcal{H}_{q}$ as an $S$-representation has finite multiplicity.
\end{thm}

\begin{proof}
	Properness of the moment map implies that, for any dominant weight $\gamma$ of $S$, the symplectic reduction $M \sslash_{\gamma} S$ is a \emph{compact} complex analytic space.
	On the other hand, $L_{\gamma}$ is a coherent sheaf on it by~\cite[\textsection~2.2]{Sja95}, and by compactness the space of sections is finite-dimensional~\cite{Car53}.
    
	It follows then from~\ref{thm:[Q,R]=0} that the irreducible representation of $S$ of highest weight $\gamma$ has finite multiplicity inside $H^{0}(M_{q}, L_{q})$, so \emph{a fortiori} inside $\mathcal{H}_{q}$.
\end{proof}

\begin{rem}
\label{rem:Hartogs}
    Another way to ensure finite-dimensionality is to assume there are compactifications of the symplectic reductions, with rational singularities and boundary of (complex) codimension at least two; then Hartogs's theorem applies on the reduction (see e.g.~\cite{Thomason} for such generalizations, and cf. Thm.~\ref{thm:formulaHHprime}).
\end{rem}

As briefly noted in the introduction, another approach to controlling the dimension of the isotypical components is offered by the results of~\cite{Wu}.
Indeed, the cited work introduces a notion of meromorphicity for certain group actions which, under appropriate conditions (see assumption~2.14 in op.~cit.), ensures finite-dimensionality.

\subsection{Rank-generating series and localisation formul\ae{}}
\label{sec:gen_series}

If either $\mathcal{H}^{(d)}$, $\mathcal{H}_{a}^{(d)}$ or $\mathcal{H}_{\lambda}^{(d)}$ have finite rank, we consider the (formal) generating series
\begin{equation}
\label{eq:H-Char}
    H(t) = \sum_{d} \rk \bigl( \mathcal{H}^{(d)} \bigr) \cdot  t^d \, ,
\end{equation}
and 
\begin{equation}
\label{eq:H'-Char}
    H' \bigl( t,\widetilde{t} \, \bigr) = \sum_{d,a} \rk \bigl(\mathcal{H}_{a}^{(d)} \bigr) \cdot t^d \widetilde{t}^a \, ,
\end{equation}
as well as
\begin{equation}
\label{eq:G-Char}
    G \bigl( t,\widetilde{t} \, \bigr) = \sum_{d\in \mathbb{Z}} \sum_{\lambda\in \Lambda^{(d)}}  \ m^{(d)}_\lambda \cdot t^d \widetilde{t}^\lambda \, .
\end{equation}

Note that if $\mathcal{H}_{a}^{(d)}$ and $\mathcal{H}_{\lambda}^{(d)}$ are both finite-rank then~\eqref{eq:H'-Char} can be obtained from~\eqref{eq:G-Char} via the substitution $\widetilde{t}^\lambda \mapsto \chi_\lambda \bigl( \, \widetilde{t} \, \bigr)$, where
\begin{equation}
    \chi_\lambda \bigl( \, \widetilde{t} \, \bigr) = \sum_{a\in E_\lambda} n^{(\lambda)}_a \cdot \widetilde{t}^a \, ,
\end{equation}
and where $E_\lambda$ is the set of weights of $V_\lambda$---with multiplicities $n^{(\lambda)}_a \in \mathbb{Z}_{\geq 0}$. 

If in particular $G_{0}$ is semisimple then the Weyl Character formula yields 
\begin{equation}
    \chi_\lambda \bigl( \, \widetilde{t} \, \bigr) =
    \frac{{\displaystyle \sum_{w\in W} \epsilon(w) \widetilde{t}^{w(\lambda + \rho)}}}
    {{\displaystyle \sum_{w\in W} \epsilon(w) \widetilde{t}^{w(\rho)}}} \, ,
\end{equation}
where $W = N(T) \big\slash T$ is the Weyl group and $\rho \in \mathfrak{t}^{\vee}$ the half-sum of positive roots.

Conversely~\eqref{eq:G-Char} can be recovered from~\eqref{eq:H'-Char} (when both are defined) as follows.
Fix $d \in \mathbb{Z}$ and let $H_{d} (\widetilde{t})$ be the coefficient of $t^{d}$ in~\eqref{eq:H'-Char}.
Let $\lambda$ be maximal among the weights such that $\widetilde{t}^{\lambda}$ appears in $H_{d}(\widetilde{t})$.
In particular, the weight $\lambda^{(0)}_{d}$ can only appear in an irreducible component of $\mathcal{H}^{(d)}_{q}$ (as a $G_{0}$-module) if it is the highest.
Therefore, the coefficient of $\widetilde{t}^{\lambda}$ in $H_{d}(\widetilde{t})$ is equal to $m_{\lambda}^{(d)}$.
One may now consider $H_{d}(\widetilde{t}) - m_{\lambda}^{(d)} \chi_{\lambda} (\widetilde{t})$ and repeat the procedure inductively.
Since each step strictly decreases one of the maximal weights the process terminates---exactly when the polynomial vanishes.
This results in a decomposition 
\begin{equation}
	H_{d} (\widetilde{t}) = \sum_{\lambda \in \Lambda^{(d)}} m_{\lambda}^{(d)} \chi_{\lambda} (\widetilde{t}) \, ,
\end{equation}
recovering all multiplicities and ultimately~\eqref{eq:G-Char}.

\vspace{5pt}

Further the generating series~\eqref{eq:H-Char},~\eqref{eq:H'-Char}, and~\eqref{eq:G-Char} can sometimes be computed by localisation formul\ae{}. 
We refer to~\cite{HW} for general results, and we review here the simpler versions used in what follows.

Suppose the action of $T_q$ on $M_q$ has a finite number of fixed points $\abs{M_q} \subseteq M_q$, and let $\overline{R}(T_q)$ be the formal completion of the character ring $R(T_q)$ of $T_q$---i.e. $\overline{R}(T_q) \simeq \mathbb{Z} \llbracket t^{\pm 1} \rrbracket$ canonically.

Since the fixed points $p \in \, \abs{M_q}$ are isolated we see that $\Lambda_{-1}(\T_p M_q) \in \overline{R}(T_q)$ is invertible.
Suppose now we have a decomposition
\begin{equation}
	H^i(M_q, L_q)  =  \overline{\bigoplus_{d \in \mathbb{Z}} H^i(M_q,L_q)^{(d)}} \, ,
\end{equation}
such that $T_q$ acts on $H^i(M_q,L_q)^{(d)}$ via the $d$'th power of the standard representation, and such that the spaces $H^i(M_q,L_q)^{(d)}$ are finite-dimensional.

\begin{prop}[\cite{Braverman,HW}]
    The following formula holds:
    \begin{equation}
        \sum_{i=0}^{2n} (-1)^i\dim H^i(M_q,L_q)^{(d)} t^d = \sum_{p \in \, \abs{M_q} } \frac{L_{q,p}}{\Lambda_{-1}(T_pM_q)} \, .
    \end{equation}
\end{prop}

Hence if $H^i(M_q,L_q) = (0) $ for $i > 0$ then simply
\begin{equation} 
\label{eq:Localize}
    H(t) = \sum_{p \in \, \abs{M_q} } \frac{L_{q,p}}{\Lambda_{-1}(T_pM_q)} \, .
\end{equation}

Considering the action of $T'_q = T_{q} \cdot T$ on $M_q$ we get an analogous result, provided $T'_q$ has finitely many fixed points and all spaces $H^i(M_q,L_q)^{(d)}_{a,q}$ are finite-dimensional---and interpreting the right-hand side as an element of $\overline{R}(T) \simeq \mathbb{Z}\llbracket t^{\pm 1}, \widetilde{t}^{\pm 1}_i \rrbracket$. 
In particular
\begin{equation} 
\label{eq:Localizeprime}
    H' \bigl( t,\widetilde{t} \, \bigr) = \sum_{p \in \, \abs{M_q} } \frac{L_{q,p}}{\Lambda_{-1}(T_pM_q)} \, .
\end{equation}

Now recall that if $M_q$ is a Stein space, or has the structure of an affine scheme, then Cartan's theorem yields the vanishing of higher cohomology groups~\cite{Cartan}.
Thus putting together the previous results we have established the following.

\begin{thm}
\label{thm:formulaHHprime}
    Suppose there exists $q \in \mathbb{C}P^1$ such that $M_q$ is a Stein space, or has the structure of an affine scheme, and that the $T_q$-action (resp. $T_q'$-action) has finitely many fixed points.
    Assume further that one of the following holds:
    \begin{itemize}
        \item there is a proper moment map for the $T_q$-action (resp. $T'_q$-action);
        
        \item there exists a compactification of the symplectic reductions with rational singularities, with boundary of codimension at least two (cf. Rem.~\ref{rem:Hartogs}).
    \end{itemize}
    Then the family $\mathcal{H}^{(d)}$ (resp. $\mathcal{H}_a^{(d)}$) has finite rank, and the associated localisation formula~\eqref{eq:Localize} (resp.~\eqref{eq:Localizeprime}) holds for the rank-generating series~\eqref{eq:H-Char} (resp.~\eqref{eq:H'-Char}).
\end{thm}

\begin{rem}
    If the higher cohomology groups do not vanish one could replace~\eqref{eq:quantum_space} by the super space
    \begin{equation}
        \widetilde{\mathcal{H}}_q = H^{\text{even}}(M_q,L) \oplus H^{\text{odd}}(M_q,L) \, ,
    \end{equation}    
    in which case case formul\ae{}~\eqref{eq:Localize} and~\eqref{eq:Localizeprime} hold for the super representations $\widetilde{\mathcal{H}}_q$ of $T_q$ and $T'_q$. 
    (In this setup one need not assume that $M_q$ be a Stein space or an affine scheme.)
\end{rem}

\begin{rem}
	Alternatively, in the setting of~\cite{Wu}, Wu's results localisation (Thm.~3.14 in op.~cit.) yield the generating series~\eqref{eq:H-Char} and~\eqref{eq:H'-Char} by an index computation of the fixed-point locus for the $T_q$- and $T'_q$-action, respectively.
\end{rem}

\subsection{\texorpdfstring{$\Sp(1)$}{Sp(1)}-symmetric hyperk\"ahler potentials }
\label{sec:HKpotSec}

\begin{defn}
    A hyperk\"ahler potential on the hyperk\"ahler manifold $(M,g,I,J,K)$ is a smooth map $\mu \colon M \to \mathbb{R}$ such that $\omega_q = i \partial_q \overline{\partial}_q \mu$ for every $q \in \mathbb{C}P^1$.
\end{defn}

One can also use such potentials to obtain equivariant pre-quantum data, as discussed below.
Assume further that $\mu$ is $\Sp(1)$-invariant and that it generates the $T_q$-actions, i.e. $i\mu \colon M_q \rightarrow i\mathbb{R} \simeq \mathfrak{t}_q^{\vee}$ is a moment map.

In this case we consider the trivial Hermitian line bundle, and lift the $G$-action by the identity on each fibre.
Natural symplectic potentials are given by
\begin{equation}
    \theta_q = \frac12 \bigl( \overline{\partial}_q \mu - \partial_q \mu \bigr) \in \Omega^1(M) \, ,
\end{equation}
hence $\nabla_q = \dif + \frac{\theta_q}{\hslash}$ is a pre-quantum connection for all $q \in \mathbb{C}P^1$, and the resulting prequantum data are $G$-equivariant since $\mu$ is $\Sp(1)$-invariant.

Now if $\operatorname{grad}(\mu)$ is complete then each $T_{q}$-action extends holomorphically to $\mathbb{C}^{\times}$, and if in addition $\mu$ is proper then the subspaces $\mathcal{H}_q^{(d)}$ are finite-dimensional by Thm~\ref{thm:fdmm}.

\begin{prop}
	\label{prop:L2H}
	Suppose that $M$ admits a $G$-invariant hyperkähler potential $\mu$ which, for every $q \in \mathbb{C}P^{1}$, is also an $\omega_{q}$-moment map for the $T_{q}$-action.
	Assume moreover that $\mu$ is bounded below and that it has finitely many critical values.
	Then for every $q \in \mathbb{C}P^{1}$ the function $\psi_{0} \coloneqq e^{-\mu/2 \hslash}$ is square-integrable and a holomorphic frame for the pre-quantum line bundle constructed above.
\end{prop}

\begin{proof}
	Nonvanishing and holomorphicity are a straightforward consequence of the definition.
	
	On the other hand, the $\operatorname{L}^{2}$-square-norm of $\psi_{0}$ can be expressed as
	\begin{equation}
		\label{eq:push_forward_int}
		\norm{\psi_{0}}_{\operatorname{L}^{2}}^{2} = \int_{M} e^{- \frac{\mu}{\hslash}} \dif \vol = \int_{B}^{\infty} e^{-\frac{\xi}{\hslash}} \mu_{*}(\dif \vol) \, ,
	\end{equation}
	where $B \in \mathbb{R}$ is a lower bound for $\mu$ and $\mu_{*} (\dif \vol)$ the push-forward of the Liouville measure.
	By the Duistermaat--Heckman theorem~\cite{DH} the push-forward admits a density which restricts to a polynomial on every interval $I \subset \mathbb{R}$ not containing critical values for $\mu$.
	Since there are finitely many such values, \eqref{eq:push_forward_int} splits as a finite sum of converging integrals, proving the claim.
\end{proof} 

By construction, the compact torus $T_q \simeq \operatorname{U}(1)$ acts on the complex vector space of holomorphic functions on $M_q$---by (inverse) pullback.
By definition, such a function is $d$-\emph{homogeneous} if it transforms (under the $T_q$-action) in the irreducible representation corresponding to the character $z \mapsto z^d \in \operatorname{U}(1)$, where $d \in \mathbb{Z}$.
Under the assumptions of Prop.~\ref{prop:L2H} we thus get an isomorphism 
\begin{equation}
    \Psi \colon \operatorname{L}^{2}H^0(M_q, \mathcal{O}, e^{-\mu/\hslash} \dif \vol)^{(d)} \longrightarrow \mathcal{H}^{(d)}_q \, ,
\end{equation}
given by $\Psi(f) = f \psi_{0}$, where the left-hand side denotes the space of $d$-homogeneous holomorphic functions with finite $\operatorname{L}^{2}$-norm with respect to $e^{-\mu/\hslash} \dif \vol$.

\section{Examples of applications}
\label{sec:examples}

\subsection{Hyperk\"ahler vector spaces}
\label{sec:HK_vector_space}

Let $n > 0$ be integer and $V$ a real vector space of dimension $4n$. 

\begin{defn}
    A linear hyperk\"ahler structure on $V$ is a scalar product $g$ and an ordered triple $(I,J,K)$ of orthogonal automorphisms of $V$ satisfying the quaternionic identities $I^2 = J^2 = K^2 = IJK = -\Id_V$.
\end{defn}

Equivalently, a linear hyperk\"ahler structure on $V$ may be regarded as a Hermitian representation of the quaternion algebra
\begin{equation}
	\mathbb{H} = \Set{ q = d + ai + bj + ck | a,b,c,d \in \mathbb{R} } \, ,
\end{equation}
on $V$, where the quaternionic Hermitian form is
\begin{equation}
	\bm{h} \coloneqq g - i \omega_{I} - j \omega_{J} - k \omega_{K} \, ,
\end{equation}
with
\begin{equation}
	\omega_{\bullet} (v,w) \coloneqq g (\bullet v , w) \quad \text{for} \quad \bullet \in \Set{I, J , K} \, .
\end{equation}

It follows that $I,J,K$ are $g$-skew-symmetric, hence they span a real Lie subalgebra $\SpanIJK \subseteq \mathfrak{o}(V,g)$.

Attached to the hyperk\"ahler vector space is the group $\Sp(V,\bm{h}) \subseteq O(V,g)$ of $\mathbb{R}$-linear endomorphisms of $V$ preserving $\bm{h}$---hence $g$ and each of the forms $\omega_I, \omega_J, \omega_K$.
As above we are interested in transformations that preserve the hyperk\"ahler structure in a looser sense, but here we restrict to \emph{linear} ones:
\begin{equation}
\label{eq:linear_hk}
	\Hk(V) = \text{Hk}(V,g,I,J,K) \coloneqq \Set{ A \in O(V,g) | \Ad_A \bigl( \SpanIJK \bigr) = \SpanIJK } \, .
\end{equation}
As a subgroup of $\operatorname{O} (V, g)$, the above is compact.

\begin{rem}
\label{rem:notation}
	We are thus slightly abusing the notation from \S~\ref{sec:setup}.
	Indeed if $V$ is regarded as a smooth hyperk\"ahler manifold then the group of \emph{all} transformations preserving $g$ and $\SpanIJK$ also contains the translations, and it is in fact generated by these two kinds of transformations.
	We shall still denote this subgroup $\Hk(V)$ in the linear case to simplify the notation.
\end{rem}

\begin{rem}
	In this case the twistor space is a rank-$2n$ holomorphic vector bundle $\pi_{\mathbb{C}P^1} \colon Z \to \mathbb{C}P^1$ isomorphic to $\underline{\mathbb{C}^{2n}} \otimes \mathscr{O}(1)$ (in the straightforward generalisation of the case $n=1$ from~\cite[Ex.~2.4, p.~143]{Hit92}). 
	The family comes with a preferred global trivialisation $Z \simeq V \times \mathbb{C}P^1$ as a smooth (rank-$4n$) vector bundle, but not as a symplectic vector bundle or as a Hermitian vector bundle.
\end{rem}

\begin{lem}
\label{lem:hk_surjection_so(3)}
    There is an exact sequence of Lie groups 
    \begin{equation}
    \label{eq:short_exact_sequence}
        1 \longrightarrow \Sp(V,\bm{h}) \longrightarrow \Hk(V) \xrightarrow{\Ad}{} \SO(3) \longrightarrow 1 \, ,
    \end{equation}
    and an embedding $\sigma \colon \Sp(1) \rightarrow \Hk(V)$ such that $\Ad \circ \, \sigma \colon \Sp(1) \to \SO(3)$ is the natural surjection. 
\end{lem}

\begin{proof}
    The natural $\Sp(1)$-action on $\mathbb{H}$ by multiplication on the right induces the standard $\Sp(1)$-action on the unit sphere of complex structures $\mathbb{S}_{IJK}$.
    The conclusion follows from a choice of identification $V \simeq \mathbb{H} \otimes_{\mathbb{R}} \mathbb{R}^n$ as $\mathbb{H}$-module.
\end{proof}

Hence a choice of orthonormal basis for $(V, \bm{h})$ (as a left $\mathbb{H}$-module) yields an identification 
\begin{equation}
	\Hk(V) = \Sp(n) \cdot \Sp(1) \simeq \bigl( \Sp(n) \times \Sp(1) \bigr) \big\slash \mathbb{Z} \slash 2 \mathbb{Z} \, .
\end{equation}
Choosing $G = \Hk(V)$, we see that in the notation of the introduction we have
\begin{equation}
    G_{0} = \Sp(V,\bm{h}) \subseteq \Hk(V) \, .
\end{equation}

\subsubsection{Geometric quantisation}

Geometric quantisation on a K\"ahler vector space is straightforward and essentially unique up to the choice of a symplectic potential, which corresponds to a gauge choice on the pre-quantum line bundle.
For $\hslash \in \mathbb{R}_{> 0}$ one considers the triple $(L,h,\nabla_q)$, consisting of the trivial complex line bundle $L \coloneqq V \times \mathbb{C} \to V$ with the tautological Hermitian metric $h$, and the connection $\nabla_q \coloneqq d - \frac{i}{\hslash}\theta_q$ defined by the invariant symplectic potential
\begin{equation}
	\theta_{q} (v) (X) = \frac{1}{2} \omega_{q} (v, X) \, ,
\end{equation}
for $v \in V$ a point and $X$ a tangent vector there.
The above yields pre-quantum data for $(V,\omega_q)$ at level $\hslash^{-1}$. 
We may denote $L_q \to V$ the line bundle to emphasize the structure we are prequantising on $V$.

The bundle $L_{q}$ comes endowed with a natural holomorphic frame
\begin{equation}
	\psi_{0} (q, v) \coloneqq \exp \left( - \frac{1}{4\hslash} g(v,v) \right) \, ,
\end{equation}
which is manifestly independent of $q \in \mathbb{C}P^1$.
For each $q$, the resulting quantum Hilbert space consists of sections $\psi = f \psi_{0}$, with $f \colon V \to \mathbb{C}$ an $I_{q}$-holomorphic \emph{function} with finite $\operatorname{L}^{2}$-norm with respect to the Gaussian measure.
This space is well known to be densely generated by the polynomial functions, which induces a grading on each $\mathcal{H}_{q}$---the Fock grading.

This setting is a particular case of the one discusssed in \S~\ref{sec:HKpotSec}.
Indeed, on a K\"ahler vector space, the function $\mu(v) = \frac{1}{2} \norm{v}^{2}$ is a moment map for the $\operatorname{U}(1)$-action by scalar multiplication and a K\"ahler potential, and moreover
\begin{equation}
	-\frac{i}{2} ( \partial - \overline{\partial}) \mu = \theta
\end{equation}
is the invariant symplectic potential.
Additionally, for each $q \in \mathbb{C}P^{1}$ the action of $T_{q}$ is the standard one.

Furthermore $d$-homogeneous holomorphic functions on a complex vector space are $d$-homogeneous polynomials, whence the decomposition of $\mathcal{H}_{q}$ into isotypical components as a $T_{q}$-module reduces to the well-known Fock grading.
By the identification of the space of such homogeneous polynomials with $\operatorname{Sym}^{d} V_{q}^{\vee}$, the finite-dimensional spaces $\mathcal{H}_{q}^{(d)}$ assemble into finite-rank Hermitian sub-bundles $\mathcal{H}^{(d)} \to \mathbb{C}P^{1}$ of the trivial $\operatorname{L}^{2}(V, L)$-bundle, with a natural isomorphism
\begin{equation}
\label{eq:iso_Sym_Hd}
	\Sym^{d} Z^{\vee} \longrightarrow \mathcal{H}^{(d)}
\end{equation}
of vector bundles over the Riemann sphere.

\subsubsection{Group action on quantum spaces}

The action $\rho^{Z} \colon \Hk(V) \to \Aut(Z)$ has a natural lift to $L = Z \times \mathbb{C}$ as $\rho^Z\times \text{Id}$.
Since $A^{*} \theta_q = \theta_{A.q}$ for $A \in \Hk(V)$ and $q \in \mathbb{C}P^1$, it follows that this action preserves the structure of $L$ as a family of pre-quantum line bundles.
This defines an action $\rho^{\mathcal{H}}$ on sections of $\mathcal{H}^{(d)}$ by pull-back, as in~\eqref{eq:pull_back_action}, and it is easy to check this is a graded fibrewise unitary $\Hk$-action---covering that on hyperk\"ahler 2-sphere.

\begin{thm}
\label{thm:CurvHKV}
    For $q \in \mathbb{C}P^1$ there is a canonical isomorphism $\mathcal{H}_q^{(d)} \simeq \Sym^d(V)$ of simple $\Sp(V,\bm{h})$-modules, and the bundle with connection $(\mathcal{H}^{(d)}, \nabla^{\mathcal{H}^{(d)}})$ is $\Hk(V)$-equivariantly isomorphic to $\mathcal{L}^{d} \otimes \underline{\Sym^d(V)} \to \mathbb{C}P^1$.
\end{thm}

\begin{proof}
	 This follows directly from the above discussion and from Thm.~\ref{thm:CurvaturePF}: The metric $g$, and hence the section $\psi_{0}$, are fixed by $\Sp(V,\bm{h})$.
	 It is known the natural action on $\Sym^{d} V^{\vee}_q$ is irreducible~\cite{Ros06}.
\end{proof}

Altogether the statements of this section establish the assumptions needed to apply Thm.~\ref{thm:super_reps}, which in this particular case yields the following.

\begin{thm}[cf. Thm.~\ref{thm:super_reps}]
\label{thm:HKvectMT}
	The $\Sp(1)$-symmetric geometric quantisation of the hyperk\"ahler vector space $V$ yields the super Hilbert space
	\begin{equation}
	    H = \overline{\bigoplus_{d \in \mathbb{Z}_{\geq 0}} H^{(d)}} \, ,
	\end{equation}
	analogously to \S~\ref{sec:super_spaces}.
	This carries a unitary $\Hk(V)$-representation preserving the splitting, and there is an isomorphism $H^{(d)} \simeq W^{(d)} \otimes \Sym^d(V)$ of simple $\Hk(V)$-modules.
	For every $d \geq 0$ we thus have
	\begin{equation}
	    \dim \bigl( W^{(d)}_+ \bigr) = (d+1) \, , \quad \dim \bigl( W^{(d)}_- \bigr) = 0 \, , \quad \dim \bigl( H^{(d)} \bigr) = (d+1)  \binom{2n+d}{d} \, .
	\end{equation}
\end{thm}

The generating series~\eqref{eq:H-Char} and~\eqref{eq:H'-Char} are obtained explicitly from the above
\begin{equation}
	H(t) = \frac{1}{(1-t)^{2n}} \, ,
	\qquad \qquad
	H' \bigl( t,\widetilde{t} \, \bigr) = \frac{1}{\prod_{i=1}^n(1-t\widetilde{t}_i)(1-t\widetilde{t}_i^{-1})} \, .
\end{equation}
On the other hand, since $V_q \simeq \mathbb{C}^{2n}$ is a Stein space, and since the actions of $T_q$ and $T'_q$ only fix the origin, Thm.~\ref{thm:formulaHHprime} also applies, and the result from~\eqref{eq:Localize} and~\eqref{eq:Localizeprime} yields the same formul\ae{}.
Now by Thm.~\ref{thm:HKvectMT} we see that $m^{(d)}_{\Sym^d(V)} = 1$ for $d \in \mathbb{Z}_{\geq 0}$, whence 
\begin{equation}
    G(t,\widetilde{t}) = \sum_{d=0}^\infty  t^d \widetilde{t}^{\lambda_{\Sym^d(V)}} \, .
\end{equation}

\subsection{Four-dimensional examples}
\label{sec:4DExample}

As mentioned in the introduction, in dimension 4 there is a complete classification of $\Sp(1)$-symmetric hyperk\"ahler manifolds up to finite quotients.
Besides $\mathbb{H}$ with its flat metric there are the Taub--NUT metrics on $\mathbb{R}^4$, and the hyperk\"ahler metric on the moduli space of charge-2 monopoles, the Atiyah--Hitchin manifold $M_{\AH}$.

\subsubsection{Taub-NUT metrics}
Consider the case of $M = \mathbb{R}^{4}$ with the Taub-NUT metric $g^{a}$ corresponding to a positive real parameter $a$---the case $a=0$ corresponds to the standard flat metric on $\mathbb{H}$, which we already discussed.
We will denote $\omega_{q}^{a}$ the corresponding symplectic structures.
It is well-known (e.g.~\cite[Rem.~1]{Gauduchon}) that
\begin{equation}
    \mHK \simeq \bigl( \Sp(1) \times \U(1) \bigr) \big\slash \mathbb{Z}_2 \simeq \U(2)\, .
\end{equation}
In particular there is a faithful $\Sp(1)$-action rotating the sphere of hyperk\"ahler structures, while $\Sp(M) = \U(1)$ is compact and commutes with $\Sp(1)$.
Furthermore there exists, unique up to isomorphism, a family of pre-quantum line bundles for $M$, since $H^{2}(M, \mathbb{Z}) = 0 = H^{1}\bigl( M, \U(1) \bigr)$.

The action of $T_{q}' = \bigl(\U(1) \times \U(1)\bigr) \slash \mathbb{Z}_{2} \subseteq \Hk(M)$ is studied explicitly by Gauduchon in~\cite[\textsection~3.2]{Gauduchon} for the complex structure $J_{+}$ corresponding to $q = i$.
The subgroup is identified in that context with $\U(1) \times \U(1)$ via the isomorphism $(t,s) \mapsto (ts, ts^{-1})$.
From Eqq.~(3.10) and~(3.19) of op.~cit. one concludes that the action of $T_{q} = \U(1) \times \{1\}$ on $M_q$ is Hamiltonian with moment map $\mu_q = \mu_{1}^{+} + \mu_{2}^{+}$ (borrowing the notation from Gauduchon), which is easily seen to be proper from the definitions.
Finally Prop.~1 (ibidem) provides a biholomorphism $\Phi^{a}_{+} = \Phi \colon (M, J_{+}) \to \mathbb{C}^{2}$, and by a straightforward check this map intertwines the $T_{q}$-action on $M_{q}$ with the standard $\U(1)$-action on $\mathbb{C}^{2}$.
In particular the $T_{q}$-action extends holomorphically to $\mathbb{C}^{*}$, and the hypotheses of Thm.~\ref{thm:fdmm} are verified.

Thus decomposing $\mathcal{H}_q$ with respect to the $T_{q}$-action yields
\begin{equation}
    \mathcal{H}_q  = \overline{\bigoplus_{d \in \mathbb{Z}_{\geq 0}} \mathcal{H}^{(d)}_q} \, ,
\end{equation}
where the subspaces $\mathcal{H}^{(d)}_q \subseteq \mathcal{H}_q$ are finite-dimensional.
Then we consider the action of the commuting compact group $\mSpo = \{ 1\} \times U(1)$ on $\mathcal{H}^{(d)}_q$ to refine:
\begin{equation}
    \label{eq:decomp_TNUT}
    \mathcal{H}^{(d)}_q  = \bigoplus_{d' \in \Lambda^{(d)}} \mathcal{H}^{(d)}_{d',q}
\end{equation}
where $\Lambda^{(d)} \subseteq \mathbb{Z}_{\geq 0}$ is finite.
In addition, we also have the following statement.

\begin{prop}
    For $q = i$, the pre-quantum line bundle $L_{q}$ admits a $T_{q}$-invariant holomorphic frame $\psi_{q}$ such that $\Phi^{*} (f) \cdot \psi_{q}$ is $\LL$ for every polynomial function $f$ on $\mathbb{C}^{2}$.
\end{prop}

\begin{proof}
    Recall that, again in the notations of~\cite{Gauduchon}, $x_{1}$, $x_{2}$, and $x_{3}$ are three real-valued functions on $M$ whose span is preserved by $\Sp(1)$, which acts on them by rotations in the standard way.
    Furthermore, all three functions are fixed by the action of $\U(1) = \Sp(M)$.
    Writing $r = \sqrt{x_{1}^{2} + x_{2}^{2} + x_{3}^{2}}$, it follows from~Eqq.~(3.19), (2.11), and (2.12) (ibidem) that the aforementioned moment map $\mu$ can be expressed as
    \begin{equation}
        \mu = r + a^{2} (x_{2}^{2} + x_{3}^{2}) \, .
    \end{equation}
    Since $\mu_{i} \coloneqq \mu$ is a moment map for $T_{i}$ with respect to $\omega_{i}^{a}$, it follows that for every $g \in \Sp(1)$ the function $(g^{-1})^{*} \mu$ generates the $T_{g.i}$-action with respect to $\omega_{g.i}^{a}$.
    In particular, if $g.i = j$, then the flow associated to
    \begin{equation}
        \mu_{j} \coloneqq (g^{-1})^{*} \mu = r + a^{2} (x_{1}^{2} + x_{3}^{2})
    \end{equation}
    with respect to $\omega_{j}^{a}$ rotates the circle spanned by $\omega_{3}^{a}$ and $\omega_{1}^{a}$.
    It is therefore a K\"ahler potential for $\omega_{i}^{a}$~\cite{HKLR87}, and repeating the argument when $g.i = k$ so is $\mu_{k} \coloneqq r + a^{2} (x_{1}^{2} + x_{2}^{2})$.
    Therefore the $T_{i}$-invariant function
    \begin{equation}
        \varphi = \varphi_{i} \coloneqq \frac{\mu_{j} + \mu_{k}}{2} = r + \frac{a^{2}}{2} (r^{2} + x_{1}^{2})
    \end{equation}
    is also a K\"ahler potential.
    It follows that for every $g \in \Sp(1)$ the function $(g^{-1})^{*} \varphi_{i}$ is completely determined by $q = g.i$, so that
    \begin{equation}
        \varphi_{q} \coloneqq (g^{-1})^{*} \varphi_{i}
    \end{equation}
    is well-defined, and a potential for $\omega_{q}^{a}$.
    From this we obtain an explicit realisation of the family of pre-quantum line bundles, for which the functions $\psi_{0,q} = e^{-\frac{1}{2\hslash} \varphi_{q}}$ define holomorphic frames.
    
    Now note the function $\mu$ is bounded below, and its only critical point is the origin---the only fixed point of the induced action.
    We may then apply the Duistermaat--Heckman theorem~\cite{DH} as in Prop.~\ref{prop:L2H} to conclude that $e^{-\alpha \mu}$ is integrable with respect to the Taub-NUT volume $\dif \vol^{a}$ for every parameter $\alpha \in \mathbb{R}_{> 0}$.
    The same clearly applies to $\mu_{j}$ and $\mu_{k}$, and from this it is easily deduced that
    \begin{equation}
        e^{-\alpha \varphi_{q}} \in \LL(M, \dif \vol^{a})
    \end{equation}
    for every $q \in \mathbb{C}P^{1}$ and $\alpha > 0$.
    In particular, the holomorphic frames constructed above are $\LL$.
    
    To conclude we recall that the two components of $\Phi$ are defined as 
    \begin{equation}
        w_{1} = e^{a^{2} x_{1}} z_{1} \, ,
        \qquad
        w_{2} = e^{-a^{2} x_{2}} z_{2} \, ,
    \end{equation}
    where $z_{1}$ and $z_{2}$ are the standard $i$-holomorphic coordinates on $M = \mathbb{H}$ with respect to the usual flat metric (cf.~\cite[Eq.~(3.4)]{Gauduchon}).
    We need to show that, for every $n, m \in \mathbb{Z}$, the function $w_{1}^{n} w_{2}^{m} \psi_{0}$ is also $\LL$.
    Expanding the definition of $\varphi$ yields
    \begin{equation}
        \begin{split}
            \frac{1}{\hslash} \varphi - 2 a^{2} (n-m) x_{1}
            ={}& \frac{r}{\hslash} + \frac{a^{2}}{2 \hslash} (r^{2} + x_{1}^{2}) - 2 a^{2} (n-m) x_{1} \\
            \geq{}& \frac{r}{2 \hslash} + \frac{a^{2}}{4 \hslash} (r^{2} + x_{1}^{2}) - C
            = \frac{1}{2 \hslash} \varphi - C \geq \frac{1}{4 \hslash} \mu_{j} - C
        \end{split}
    \end{equation}
    provided $C \in \mathbb{R}_{> 0}$ is large enough.
    Furthermore, it is a simple consequence of the definitions in~\cite{Gauduchon} that $\abs{z_{1}}^{2} + \abs{z_{2}}^{2} = 2r$, whence
    \begin{equation}
        \abs{z_{1}^{n} z_{2}^{m}}^{2} \leq (2r)^{2(n+m)} \leq \mu_{j}^{2(n+m)} \, .
    \end{equation}
    Collecting the estimates and using again the Duistermaat--Heckman theorem we conclude:
    \begin{equation}
        \int_{M} \abs{w_{1}^{n} w_{2}^{m}}^{2} \psi_{0}^{2} \dif \vol^{a}
        \leq e^{C} \int_{M} \mu_{j}^{2(m+n)} e^{-\frac{1}{4\hslash} \mu_{j}} \dif \vol^{a} < \infty \, . 
    \end{equation}
\end{proof}

As a consequence of this result we have $\dim \mathcal{H}^{(d)}_{d'} = 1$ for all $d \in \mathbb{Z}_{\geq 0}$ and $d' = d-2j$ with $j \in \set{0,\dotsc,d}$, and we conclude that $\mathcal{H}^{(d)}_{d'} \simeq \mathcal{L}^{d}$ for such values of $d$ and $d'$.

\begin{thm}
    The generating series~\eqref{eq:H'-Char} is the same as for the flat metric, namely
    \begin{equation}
        H' \bigl( t,\widetilde{t} \, \bigr) = \frac{1}{\bigl( 1 - t \widetilde{t} \, \bigr) \bigl( 1 - t\widetilde{t}^{-1} \, \bigr)} \, .
    \end{equation}
    
    We thus have
    \begin{equation}
        H = \overline{\bigoplus_{d \in \mathbb{Z}_{\geq 0}} H^{(d)}} \, , \qquad H^{(d)} = \bigoplus_{d' \in \Lambda^{(d)}} H_{d'}^{(d)} = H^{0} \left(\mathbb{C}P^{1}, \mathcal{L}^{d} \right)^{\oplus(d+1)} \, .
    \end{equation}
\end{thm}

\subsubsection{The Atiyah--Hitchin manifold}

Let us consider the Atiyah--Hitchin manifold $M_{\AH}$, the last four-dimensional case.
We shall discuss the extent to which our methods apply here.

The Atiyah--Hitchin manifold can be realised as the moduli space of charge-2 centred magnetic monopoles in $\mathbb{R}^{3}$, and it comes with a natural Riemannian metric preserved by the $\SO(3)$-action induced by rotating monopoles.
The quaternionic nature of the Bogomolny equation, of which the monopoles represented by $M_{\AH}$ are a particular class of solutions, induces a family of almost complex structures, which can be better understood via Donaldson's description in terms of rational maps~\cite{Don84b}.
More precisely, the choice of an oriented line through the origin in $\mathbb{R}^3$ induces an identification
\begin{equation}
    \widetilde{M}_{\AH} = \Set{S(z) = \frac{uz + v}{z^{2} - w} \in \mathbb{C}(z) | v^{2} - wu^{2} = 1} \eqqcolon R_{2}^{0} \, ,
\end{equation}
where the left-hand side denotes the (two-fold) universal cover of $M_{\AH}$.
The Atiyah--Hitchin manifold is recovered from the monodromy action, generated by $(u,v,w) \mapsto (-u,-v,w)$.
The resulting map is a biholomorphism with respect to one of the aforementioned almost complex structures, establishing that the latter is integrable and the former is K\"ahler.
Rotations around the preferred direction induce a $\U(1)$-action of $R_{2}^{0}$ by
\begin{equation}
    \label{eq:MAHstab}
    t.(u,v,w) = (tu, v, t^{-2}w) \, .
\end{equation}
As the preferred direction changes across all possible choices, this results in a family of K\"ahler structures parametrised by $\mathbb{C}P^{1}$, which is clearly rotated by the $\SO(3)$-action (see~\cite[Ch.~2]{AHBook}).

The above identification is not isometric with respect to the Riemannian embedding $R_{2}^{0} \subseteq \mathbb{C}^{3}$; nonetheless, the Riemannian structure on $M_{\AH}$ can be described by studying the $\SO(3)$-orbits~\cite[Chs.~8-11]{AHBook}.
The generic stabiliser of a monopole is the Klein four-group $K_{4}$, while orbits are parametrised by $k = \sin(\alpha)$ for an angle $\alpha \in \bigl[ 0,\frac{\pi}{2} \bigr]$, resulting in a description of an open dense of $M_{\AH}$ as the product $(0, 1) \times \SO(3) \slash K_{4}$; further as $k \to 0$ the orbit degenerates to a diffeomorphic copy of $\mathbb{R}P^{2}$, onto which $M_{\AH}$ deformation-retracts.

According to Swann's work~\cite[\textsection~6, \emph{Four-manifolds}]{Swann}, $M_{\AH}$ does not admit a hyperk\"ahler potential.
Furthermore, one sees from~\eqref{eq:MAHstab} that the stabiliser of each K\"ahler structure has exactly one fixed point, and since the manifold has the homotopy type of $\mathbb{R}P^{2}$ there can be no proper moment map.
Nonetheless the above homotopy equivalence yields 
\begin{equation}
    H^{1} \bigl( M_{\AH} , \U(1) \bigr) \simeq H^{2} (M_{\AH}, \mathbb{Z}) \simeq \mathbb{Z} \big\slash 2\mathbb{Z} \, ,
\end{equation}
hence by \S~\ref{sec:Sp1equi} there are exactly two inequivalent $\SO(3)$-equivariant families of pre-quantum line bundles.
They differ by a twist by a family of flat connections on the non-trivial complex line bundle on $M_{\AH}$.

The family supported on the trivial bundle can be constructed by means of the K\"ahler potentials of Olivier~\cite{Oli91}.
Namely the metric on the Atiyah--Hitchin manifold is the completion of
\begin{equation}
    \label{eq:AHmetric}
    \dif s^{2} = \frac{\beta^{2} \gamma^{2} \delta^{2}}{\left(4 k^{2} (1-k^{2}) K^{2}\right)^{2}} \dif m^{2} + \beta^{2} \sigma_{x}^{2} + \gamma^{2} \sigma_{y}^{2} + \delta^{2} \sigma_{z}^{2} \, ,
\end{equation}
defined on $(0, \frac{\pi}{2}) \times \SO(3) \slash K_{4}$.
We follow the conventions of op.~cit.
Namely, $m = k^{2}$ is used as a coordinate in place of $k$, while $(\sigma_{x}, \sigma_{y}, \sigma_{z})$ is an orthonormal frame of $\T^*\SO(3) \to \SO(3)$ and the coefficients $\beta, \gamma, \delta$ are functions of $k$ determined by
\begin{equation}
    \beta \gamma = - E K \, ,
    \qquad
    \gamma \delta = - E K + K^{2} \, ,
    \qquad
    \beta \delta = - E K + (1-k^{2}) K^{2} \, ,
\end{equation}
where
\begin{equation}
    K \coloneqq K(k) = \int_{0}^{\frac{\pi}{2}} \frac{\dif \phi}{\sqrt{1 - k^{2} \sin^{2}\phi}} \, ,
    \qquad
    E \coloneqq E(k) = \int_{0}^{\frac{\pi}{2}} \sqrt{1 - k^{2} \sin^{2} \phi} \dif \phi
\end{equation}
are the complete elliptic integrals of the first and second kind, respectively.

Op.~cit. then uses the Euler angles $(\varphi, \theta, \psi)$ as coordinates on $\SO(3)$ to give an explicit K\"ahler potential $\Omega$ for one of the complex structures, say $I_{3}$, preserved by rotations in the angle $\varphi$.
This is given in of~\cite[Eq.~55]{Oli91} and can be written explicitly using Eqq.~(6), (24), (25) and (36) therein, getting the formula
\begin{equation}
    \Omega = \frac{\beta \gamma + \gamma \delta + \delta \beta}{8} + \frac{1}{8} \left(\gamma \delta \sin^{2} \theta \cos^{2} \psi + \delta \beta \sin^{2} \theta \sin^{2} \psi + \gamma \beta \cos^{2} \theta \right) \, .
\end{equation}
Note for $k \in (0,1)$ this function extends continuously to the whole of $\SO(3)$, and the trigonometric functions of $(\theta,\psi)$ descend to the projective space at $k=0$; hence the potential extends to the completion $M_{\AH}$.
Finally, we emphasise that this potential is independent of the variable $\varphi$, which is to say that it is invariant under the action of the $I_{3}$-stabiliser.
It follows that $\Omega$ defines an equivariant family of potentials under the $\SO(3)$-action, whence an equivariant family of pre-quantum line bundles by the usual construction, together with a holomorphic frame $\psi_{0} = e^{-\frac{1}{2 \hslash} \Omega}$ for $I_{3}$.

\begin{prop}
    \label{prop:L2AH}
    The function $e^{- \alpha \Omega}$ is integrable on $M_{\AH}$ for $\alpha \in \mathbb{R}_{> 0}$.
\end{prop}

\begin{proof}
    From~\eqref{eq:AHmetric} we obtain the following expression for the volume form on (the complement of a negligible set in) $M_{\AH}$:
    \begin{equation}
        \dif \vol = \frac{\beta^{2} \gamma^{2} \delta^{2}}{4 k^{2} (1-k^{2}) K^{2}} \, \dif m \sigma_{x} \sigma_{y} \sigma_{z} \, .
    \end{equation}
    We need to show that
    \begin{equation}
        \int_{(0,1) \times \SO(3)} e^{- \alpha \Omega} \frac{\beta^{2} \gamma^{2} \delta^{2}}{4 k^{2} (1-k^{2}) K^{2}} \, \dif m \sigma_{x} \sigma_{y} \sigma_{z} < \infty \, .
    \end{equation}
    
    Note that $\beta \gamma \leq 0$, $\gamma \delta \geq 0$, and $\beta \delta \leq 0$ yield
    \begin{equation}
        \Omega \geq \frac{\gamma \delta}{8} \, .
    \end{equation}
    We may then use these bounds and the Fubini-Tonelli theorem to reduce the statement to
    \begin{equation}
        \int_{0}^{1} e^{-\frac{\alpha}{8} \gamma \delta} \frac{\beta^{2} \gamma^{2} \delta^{2}}{k^{2} (1-k^{2}) K^{2}} \, \dif m < \infty \, .
    \end{equation}
    
    We will proceed by studying the asymptotic behaviour of the integrand in the limit $k \to 1$---the integral is necessarily regular for $k \to 0$.
    It is well-known that
    \begin{equation}
        K \sim \frac{1}{2} \log (1 - k^{2}) \, ,
    \end{equation}
    and since $E(1) = 1$ we find that
    \begin{equation}
        \beta \gamma \sim - \frac{1}{2} \log(1 - k^{2}) \, ,
        \qquad
        \gamma \delta \sim \frac{1}{4} \log^2(1 - k^{2}) \, ,
        \qquad
        \beta \delta \sim - \frac{1}{2} \log(1-k^{2}) \, ,
    \end{equation}
    hence the integral converges by comparison with
    \begin{equation}
        \int_{0}^{1} \exp \Bigl( -\frac{\alpha}{32} \log^2(1-k^{2}) \Bigr) \frac{\log^2(1-k^{2})}{(1-k^{2})} \dif m = \int_{0}^{\infty} e^{-\frac{\alpha}{32} x^{2}} x^{2} \dif x < \infty \, .
    \end{equation}
\end{proof}

For $\alpha = 1/\hslash$ this implies the holomorphic frame $\psi_{0}$ is $\LL$, hence an element of $\mathcal{H}^{(0)}_{I_{3}}$; in principle more $\LL$ holomorphic sections may be found considering functions of the holomorphic coordinates $(u,v,w)$ on $R_{0}^{2}$.
If all the monomials that descend to $M_{\AH}$ are $\LL$, then one concludes that $\mathcal{H}_{q}^{(d)}$ has infinite rank for every integer $d$, since $u^{a} v^{b} w^{c}$ is $(a - 2c)$-homogeneous and well-defined on $M_{\AH}$ if $a+b$ is even.
We obtain a partial result in this direction, showing that all powers of $w$ are $\LL$.

The problem of describing $u$, $v$, and $w$ in terms of the setup above is addressed in~\cite[Chapters 6-7]{AHBook}, by making use of the twistor description and spectral curves~\cite{Hur83}.
Introducing parameters
\begin{equation}
    k_{1} = \frac{\sqrt{k \sqrt{1-k^{2}}} K}{2} \, ,
    \qquad
    k_{2} = \frac{1 - 2k^{2}}{3 k \sqrt{1-k^{2}}} \, ,
\end{equation}
consider the elliptic curve
\begin{equation}
    y^{2} = 4k_{1}^{2} \bigl( x^{3} - 3k_{2} x^{2} - x \bigr)
\end{equation}
and let $\wp$, $\zeta$ be its corresponding Weierstrass functions, $\eta$ the real period of $\zeta$.
Suppose that $a, b \in \mathbb{C}$ are the entries of a matrix in $\SU(2)$, thought of as a parametrisation of $\SO(3)/K_{4}$, and let $\xi \in \mathbb{C}$ be such that
\begin{equation}
    \label{eq:implicitxi}
    \wp(\xi) = \frac{b}{\overline{a}} - k_{2} \, .
\end{equation}
Then the corresponding point in $M_{\AH}$ has holomorphic coordinates
\begin{equation}
    \begin{aligned}
        u ={}& \frac{\sinh \bigl( 2 k_{1} \zeta(\xi) - \frac{\eta \xi}{2} + k_{1} \overline{a} \overline{b} \wp'(\xi) \bigr)}{k_{1} \overline{a}^{2} \wp'(\xi)} \, , \\[.4em]
        v ={}& \cosh \bigl( 2 k_{1} \zeta(\xi) - \frac{\eta \xi}{2} + k_{1} \overline{a} \overline{b} \wp'(\xi) \bigr) \, , \\[.4em]
        w ={}& k_{1}^{2} \overline{a}^{4} \wp'(\xi)^{2} \, ,
    \end{aligned}
\end{equation}
up to the sign ambiguity resulting from the monodromy.
Substituting~\eqref{eq:implicitxi} in the differential equation for $\wp$, and using $g_{2}$ and $g_{3}$ as given in~\cite{Hur83}, we obtain
\begin{equation}
    w = k_{1}^{2} \overline{a} \bigl( - 12 \overline{a}b^{2} k_{2} + 4b^{3} - 4 \overline{a}^{2} b \bigr) \, .
\end{equation}
Now since $\abs{a}^{2} + \abs{b}^{2} = 1$ a straightforward check shows that
\begin{equation}
    \abs{w}^{2} \leq 16 k_{1}^{4} (9k_{2}^{2} + 2) \sim 4 K^{2} \sim \log^2(1-k^{2}) \, ,
\end{equation}
for $k \to 1$.
Adapting the proof of Prop.~\ref{prop:L2AH} and using~\eqref{eq:MAHstab} we obtain the following.

\begin{prop}
    For every integer $n \geq 0$ the holomorphic section $w^{n} \psi_{0}$ is $\LL$ and therefore an element of $\mathcal{H}^{(-2n)}_{I_{3}}$.
\end{prop}

The analysis is more delicate for the functions $u$ and $v$.
Using~\eqref{eq:implicitxi} one can express $\overline{a} \overline{b}$ in terms of $\xi$ and write the argument of the hyperbolic functions as
\begin{equation}
    \Phi(\xi)
    = 2 k_{1} \zeta(\xi) - \frac{\eta \xi}{2} + k_{1} \wp'(\xi) \frac{k_{2} + \overline{\wp(\xi)}}{1 + \abs{ k_{2} + \wp(\xi)}^{2}} \, .
\end{equation}
It follows from the definitions and the Legendre relation that this function is periodic for the real period of $\wp$ and quasi-periodic for the imaginary period, with step $\pi i$, whence the sign ambiguity of $u$ and $v$.
Moreover one can show the poles of the summands cancel out, leaving a non-holomorphic analytic function---hallmark of the fact that the $\SO(3)$-action does not preserve the complex structure.
In particular its real part is bounded for fixed ``$k$''.

\subsection{Moduli spaces of framed \texorpdfstring{$\SU(r)$}{SU(r)}-instantons}
\label{sec:Instantons}

Let $r \geq 2$ and $k \geq 0$ be integers, and consider the moduli space $M_{k,r}$ of charge-$k$ framed $\SU(r)$-instantons on $\mathbb{R}^{4}$, which is a hyperk\"ahler manifold~\cite{Atiyah,Donaldson}.
Each of its complex structures can be described in terms of the ADHM construction as follows, after fixing an identification $\mathbb{R}^{4} \simeq \mathbb{C}^{2}$.
Consider the product 
\begin{equation}
    \mathbb{M} \coloneqq \End \bigl( \mathbb{C}^k \bigr)^2 \times \Hom \bigl( \mathbb{C}^k,\mathbb{C}^r \bigr) \times \Hom \bigl( \mathbb{C}^r,\mathbb{C}^k \bigr) \, ,
\end{equation}
with $\GL \bigl( \mathbb{C}^k \bigr)$-action given by
\begin{equation}
    g.(\alpha_0,\alpha_1,a,b) = \bigl( g \alpha_0 g^{-1}, g \alpha_1 g^{-1}, ga, bg^{-1} \bigr) \, .\footnote{Note $\mathbb{M}$ is a space of representations of a quiver on two nodes---having a double loop-edge at one node, and a pair of opposite arrows between the two nodes---and that the action naturally extends to $\GL(\mathbb{C}^k) \times \GL(\mathbb{C}^r)$ (which controls isomorphisms of representations).}
\end{equation}
Let $\mathbb{M}_{0}$ denote the set of elements of $\mathbb{M}$ satisfying the additional conditions
\begin{enumerate}[(i)]
    \item \label{cond:cut} $[\alpha_{0}, \alpha_{1}] + ab = 0$;
    \item \label{cond:stability} For all $\lambda, \mu \in \mathbb{C}$, $\begin{pmatrix} \alpha_{0} + \lambda \\ \alpha_{1} + \mu \\ a \end{pmatrix}$ is injective and $\begin{pmatrix} \lambda - \alpha_{0} & \alpha_{1} - \mu & b \end{pmatrix}$ is surjective.
\end{enumerate}
Then the restricted $\U(r)$-action is Hamiltonian with moment map
\begin{equation}
    \mu(\alpha_{0}, \alpha_{1}, a, b) \coloneqq [\alpha_{1}, \alpha_{1}^{*}] + [\alpha_{2}, \alpha_{2}^{*}] + bb^{*} - a^{*}a \, ,
\end{equation}
and there is an identification
\begin{equation}
    \label{eq:ADHM}
    M_{k,r} \simeq \mathbb{M}_{0} \sslash_{\mu} \U(k) \, .
\end{equation}

The rotation group $\SO(4)$ acts on $M_{k,r}$, and in particular the subgroup $\Sp(1)$, in the identification $\mathbb{R}^{4} \simeq \mathbb{H}$, transitively permutes the complex structures.
Furthermore, the work of Maciocia~\cite{Maciocia} shows that for each $q \in \mathbb{C}P^{1}$ the $T_{q}$-action has moment map
\begin{equation}
    m_{2} (A) = \frac{1}{16 \pi^{2}} \int_{\mathbb{R}^{4}} \norm{x}^{2} \, \operatorname{tr} F_{A}^{2} \, .
\end{equation}
This function is clearly $\Sp(1)$-invariant, and therefore a hyperkähler potential, so one can construct an $\Sp(1)$-invariant family of pre-quantum line bundles endowed with holomorphic frames as in \S~\ref{sec:HKpotSec}.

The function $m_{2}$ is not, however, a proper map.
By op.~cit., under the identification~\eqref{eq:ADHM} it corresponds to the norm-squared function $f \colon \mathbb{M}_{0} \to \mathbb{R}$, which is $\U(k)$-invariant but not proper, on account of the open condition~\ref{cond:stability}.
However Donaldson's work~\cite{Donaldson} identifies the symplectic reduction~\eqref{eq:ADHM} with the GIT quotient of $\mathbb{M}_{0}$ by $\GL(k,\mathbb{C})$, whereupon~\ref{cond:stability} translates into a stability condition.
One may then include the semi-stable points to obtain a partial compactification
\begin{equation}
    \overline{M}_{k,r} \coloneqq \mathbb{M} \sslash_{\mathrm{GIT}} \GL(k,\mathbb{C}) \, ,
\end{equation}
which is smooth by the work of Nakajima and Yoshioka~\cite[Cor.~2.2]{NakYosh}.
The map $f$ descends then to a proper one on $\overline{M}_{k,r}$; it is also clear that its gradient is complete on the quotient, showing that geometric quantisation on this space yields finite-rank isotypical components by Thm.~\ref{thm:fdmm}.
On the other hand, the codimension of the boundary $\overline{M}_{k,r} \setminus M_{k,r}$ is greather than $2$, so that Hartogs's theorem allows for the extension of holomorphic functions on $M_{k,r}$, which yields the finite-dimensionality of the isotypic components over this latter space.

\section{Outlook and further perspectives}
\label{sec:further}

There are more spaces that fit some of the requirements for our quantisation scheme.

\vspace{5pt}

By the work of Kronheimer~\cite{Kronheimer} the nilpotent (co)adjoint orbits of complex semisimple (1-connected) Lie groups are hyperk\"ahler manifolds with transitively permuting $\SO(3)$-actions, and by Swann's work~\cite{Swann} they admit hyperkähler potentials.
Indeed, Prop.~5.5 of op.~cit. states that such a potential exists on a hyperkähler manifold if it admits an $\Sp(1)$-action permuting the complex structure and such that, calling $X_{q}$ the vector field generating the $T_{q}$-action for each $q \in \mathbb{C}P^{1}$, the vector field $I_{q} X_{q}$ is independent of $q$.
Following Prop.~6.5 of op.~cit., Swann goes on to check that this condition is verified for Kronheimer's space, thus establishing the existence of a hyperkähler potential.
This is a particular instance of hyperk\"ahler moduli spaces of solutions of Nahm's equations, specifically on a half-line with nilpotent boundary conditions.\footnote{The hyperk\"ahler metric on general orbits was constructed in~\cite{Biq96,Kov96}.}
Since Nahm's equations come naturally with a quaternionic structure and $\Sp(1)$-action, the resulting manifolds have symmetries of the kind considered in this paper, and different choices of domain and boundary conditions give rise to different hyperk\"ahler structures.
For instance, semisimple boundary conditions on a half-line result in orbits of semisimple elements~\cite{Kro90b}, while the study of Nahm's equations on a compact interval leads to the cotangent bundle $\operatorname{T}^{*} G$~\cite{Kro88,DS96}.
By the works of Mayrand~\cite{May19,May19a,May20}, the latter comes with natural $\Sp(1)$-equivariant families of K\"ahler potentials and moment maps for the stabilizers $T_{q}$, rather than a hyperk\"ahler one, and they enjoy interesting properties that might lead to a variation of our main construction.

Also, as mentioned in the introduction, many new interesting hyperk\"ahler metrics can be defined on moduli spaces of irregular connections/Higgs bundles over (wild generalisations of) Riemann surfaces~\cite{BB04,Wit08}, with simple examples reviewed in~\cite{Boa18}: the ``multiplicative'' versions of the Eguchi--Hanson space and Calabi's examples (whose standard ``additive'' versions are quiver varieties on two nodes).
This fits into a more general (new) multiplicative theory of quiver varieties~\cite{Boa15}, involving a ``fission'' gluing operation generalising the TQFT construction of moduli spaces of flat connections~\cite{Boa07,Boa14}; note that conjecturally this produces a lot more new hyperk\"ahler manifolds~\cite{Boa09}, beyond (wild) nonabelian Hodge spaces.
See~\cite{Rem19,Rem20,FelRem} about quantum moduli spaces of meromorphic connections. 

Finally the example of \S~\ref{sec:Instantons}, i.e. the moduli spaces of framed $\SU(r)$-instantons, opens the way for further discussion on the relation between the generating series produced by this new quantisation scheme and the well-known Nekrasov partition functions.

\appendix
\section{Comparison with the standard approach}
\label{sec:appendix}

In this section we shall correct the family of quantum Hilbert spaces $\mathcal{H}_q$ to obtain finite-rank flat vector bundles of isotypical components (under the main assumption), as well as unitary equivalences between the quantisation of $M$ with respect to the given K\"ahler polarisations.

Based on Thm.~\ref{thm:CurvaturePF}, we do this by a correcting twist of the finite-rank bundles $\mathcal{H}^{(d)}_\lambda \to \mathbb{C}P^1$; namely consider the tensor product
\begin{equation}
	\widetilde{\mathcal{H}}_{\lambda}^{(d)} \coloneqq \mathcal{H}_{\lambda}^{(d)} \otimes \mathcal{L}^{-d} \, , \qquad \text{for} \qquad d \in \mathbb{Z} \, , \lambda\in\Lambda^{(d)} \, .
\end{equation}
This new vector bundle comes with a $\mpHKo$-action, and we denote $\nabla^{\widetilde{\mathcal{H}}_{\lambda}^{(d)}}$ the resulting $\mpHKo$-invariant flat connection. 

Since $\mathbb{C}P^1$ is simply-connected the parallel transport defines canonical unitary isomorphisms 
\begin{equation}
\label{eq:unitary_cocycle}
    \widetilde{\mathcal{H}}^{(d)}_{q,\lambda} \longrightarrow \widetilde{\mathcal{H}}^{(d)}_{q',\lambda} \, , \qquad \text{for } q, q' \in \mathbb{C}P^1 \, ,
\end{equation}
satisfying 1-cocycle identities.
Analogously to the above we then define
\begin{equation}
    \overline{\bigoplus_{\lambda \in \Lambda^{(d)}} \widetilde{\mathcal{H}}_{q,\lambda}^{(d)}} \eqqcolon \widetilde{\mathcal{H}}_q^{(d)} \subseteq \widetilde{\mathcal{H}}_q \coloneqq \overline{\bigoplus_{d \in \mathbb{Z}} \widetilde{\mathcal{H}}^{(d)}_q} \, ,
\end{equation}
and these families of Hilbert spaces carry a 1-cocycle of unitary isomorphisms induced from~\eqref{eq:unitary_cocycle}: This is the usual geometric quantisation construction.

Now we can introduce super Hilbert spaces $\widetilde{H}^{(d)}_{\lambda,j}$ analogously to \S~\ref{sec:super_spaces}, taking the holomorphic cohomology of the twisted vector bundles $\widetilde{\mathcal{H}}^{(d)}_{\lambda} \to \mathbb{C}P^1$.

\begin{thm}[cf. Thm.~\ref{thm:super_reps}]
    There is a unitary action $\mpHKo \rightarrow \U \bigl( \widetilde{H} \bigr)$
    preserving the nested splittings:
    \begin{equation}
        \widetilde{H} \coloneqq \overline{\bigoplus_{d \in \mathbb{Z}} \widetilde{H}^{(d)}} \, , \quad \widetilde{H}^{(d)} \coloneqq \overline{\bigoplus_{\lambda \in \Lambda^{(d)}} \widetilde{H}_\lambda^{(d)}} \, , \quad  \widetilde{H}_\lambda^{(d)} \coloneqq \bigoplus_{j = 1}^{m_{\lambda}^{(d)}} \widetilde{H}^{(d)}_{\lambda,j} \, .
    \end{equation}
\end{thm}

Finally we can compare this representation with the one constructed in \S~\ref{sec:super_spaces}, finding that twisting trivializes part of the action.
Namely, the present super Hilbert space $\widetilde{H}_{\lambda,j}^{(d)} \simeq  V_\lambda \otimes W^{(0)}$ replaces the original $H_{\lambda,j}^{(d)} \simeq V_\lambda \otimes W^{(d)}$ as a $\mpHKo$-module---recalling that $W^{(0)}$ is the trivial one-dimensional $\Sp(1)$-module. 
This should be compared with the (more) interesting irreducible representations of $\mpHKo$ obtained from Thm.~\ref{thm:super_reps}.

\todos

\end{document}